\pdfoutput=1
\RequirePackage{ifpdf}
\ifpdf 
\documentclass[pdftex]{sigma}
\else
\documentclass{sigma}
\fi

\usepackage{tikz}
\usetikzlibrary{matrix,arrows,calc,cd}

\newtheorem{thm}{Theorem}
\newtheorem{lem}[thm]{Lemma}
\newtheorem{prop}[thm]{Proposition}
\newtheorem{cor}[thm]{Corollary}

\theoremstyle{definition}
\newtheorem{ex}[thm]{Example}
\newtheorem{rem}[thm]{Remark}

\newcommand{\Z}{\mathbb{Z}}
\newcommand{\Q}{\mathbb{Q}}

\newcommand{\C}{\mathbb{C}}

\renewcommand{\P}{\mathbb{P}}
\newcommand{\G}{\mathbb{G}}

\newcommand{\into}{\hookrightarrow}

\newcommand{\xto}[2]{\xrightarrow[#1]{#2}}
\newcommand{\ot}{\leftarrow}
\newcommand{\xot}[2]{\xleftarrow[#1]{#2}}

\newcommand{\OO}{\mathcal{O}}

\newcommand{\VV}{\mathcal{V}}

\renewcommand{\phi}{\varphi}
\renewcommand{\epsilon}{\varepsilon}

\DeclareMathOperator{\Hom}{Hom}

\DeclareMathOperator{\GL}{GL}

\DeclareMathOperator{\id}{id}

\DeclareMathOperator{\im}{im}

\DeclareMathOperator{\pt}{pt}

\DeclareMathOperator{\sst}{sst}
\DeclareMathOperator{\st}{st}

\DeclareMathOperator{\dimvect}{\underline{\mathrm{dim}}}

\newcommand{\one}{\mathbf{1}}

\newcommand{\sub}{\subseteq}

\renewcommand{\hat}[1]{\widehat{#1}}

\begin{document}
\allowdisplaybreaks

\newcommand{\arXivNumber}{1911.03288}

\renewcommand{\PaperNumber}{096}

\FirstPageHeading

\ShortArticleName{Torus-Equivariant Chow Rings of Quiver Moduli}

\ArticleName{Torus-Equivariant Chow Rings of Quiver Moduli}

\Author{Hans FRANZEN}

\AuthorNameForHeading{H.~Franzen}

\Address{Fakult\"at f\"ur Mathematik, Ruhr-Universit\"at Bochum,\\ Universit\"atsstra{\ss}e 150, 44780 Bochum, Germany}
\Email{\href{mailto:hans.franzen@rub.de}{hans.franzen@rub.de}}
\URLaddress{\url{www.ruhr-uni-bochum.de/ffm/Lehrstuehle/Algebra/franzen.html.en}}

\ArticleDates{Received March 14, 2020, in final form September 16, 2020; Published online September 30, 2020}

\Abstract{We compute rational equivariant Chow rings with respect to a torus of quiver moduli spaces. We derive a presentation in terms of generators and relations, use torus localization to identify it as a subring of the Chow ring of the fixed point locus, and we compare the two descriptions.}


\Classification{14C15; 16G20}

\section{Introduction}

In this paper we study actions of tori on moduli spaces of quiver representations and determine their rational equivariant Chow rings as introduced by Edidin and Graham in \cite{EG:98}. More precisely, given a quiver $Q$, a dimension vector $d$, and a stability condition $\theta$, we look at the action of \mbox{$T = \G_m^{Q_1}$} which acts on $M^{\theta-\st}(Q,d)$ by scaling along the arrows of $Q$. This action was introduced by Weist in \cite{Weist:13} and has since been used successfully, for instance in \cite{RSW:12, RW:13, RW:20}. The main objective of this paper is to determine the pull-back $i^*\colon  A_T^*\big(M^{\theta-\st}(Q,d)\big)_\Q \to A_T^*\big(M^{\theta-\st}(Q,d)^T\big)_\Q$ of the embedding $i\colon M^{\theta-\st}(Q,d)^T \to M^{\theta-\st}(Q,d)$ of the fixed point locus explicitly and hence to find the $T$-equivariant Chow ring of $M^{\theta-\st}(Q,d)$ as a subring of $A^*\big(M^{\theta-\st}(Q,d)^T\big)_\Q \otimes A_T^*(\pt)_\Q$. This is called torus localization and it is an idea that goes back to Goresky--Kottwitz--MacPherson's paper \cite{GKM:98} and even further to Chang--Skjelbred~\cite{CS:74}, where equivariant cohomology rings of equivariantly formal spaces with respect to a compact torus are studied. Brion developed in~\cite{Brion:97} the algebraic counterpart of this theory. He gives descriptions of the equivariant Chow ring for an action of an algebraic torus on a variety. Both articles provide results of various generalities depending on the nature of the action of the torus. The most explicit description is possible if there are only finitely many fixed points and finitely many one-dimensional orbits. For quiver moduli however, this is hardly ever fulfilled. The fixed points of the aforementioned torus action are usually not isolated and even if they are, there might be infinitely many one-dimensional orbits. Still there is one particularly easy class of quiver moduli in which this condition holds. If we assume that the entries of~$d$ are all~$1$ then the moduli space is toric. This toric variety is well understood, see~\cite{AH:99}. We study this case separately.

One of the main results of \cite{Weist:13} asserts that each connected component of the locus of $T$-fixed points of $M^{\theta-\st}(Q,d)$ is isomorphic to a stable moduli space of the universal abelian covering quiver $\smash{\hat{Q}}$ of $Q$. This result is the key ingredient to determine the localization map. We use that $M^{\theta-\st}(Q,d)$ is an algebraic quotient $R(Q,d)^{\theta-\st}/PG_d$ by a reductive group to identify the $T$-equivariant Chow ring of the quotient with the $PG_d \times T$-equivariant Chow ring of the total space. We can adapt a method of~\cite{FR:18:Sst_ChowHa} to obtain a presentation of this ring in terms of Chern roots of certain $PG_d \times T$-equivariant vector bundles. This is Theorem~\ref{t:taut}. We then exhibit in Theorem~\ref{t:main} the images of these generators under the localization map $i^*$ using Weist's characterization of the fixed points. This description is most useful if the torus acts with finitely many fixed points. We illustrate this result in an example where there are finitely many fixed points but infinitely many one-dimensional orbits. After that we turn to the case where $d=\one$, i.e., $d$ consists of ones entirely. In this case, $PG_\one$ is itself a torus which embeds into~$T$. We study the action of the cokernel~$T_0$ on the moduli space. We determine the fixed points and the one-dimensional orbits. Our description of the image of the localization map then follows from a result of Brion~\cite{Brion:97}.

\section{Generalities on quiver moduli} \label{s:genl}

Fix an algebraically closed field $k$. Let $Q$ be a quiver. We denote its set of vertices with $Q_0$ and its set of arrows with $Q_1$. With $s(\alpha)$ and $t(\alpha)$ we denote the source and target of an arrow $\alpha$. We assume throughout that $Q$ is connected, which means that in the underlying unoriented graph, all vertices are connected by a path. Furthermore, we suppose that for every two vertices $i,j \in Q_0$ there are just finitely many $\alpha \in Q_1$ such that $s(\alpha) = i$ and $t(\alpha) = j$. If $Q_0$ (and hence also $Q_1$) is finite, then we say that $Q$ is a finite quiver.

We assume that the reader is familiar with the basics of representation theory of quivers. For an introduction to the subject we refer to \cite[Section~II.1]{ASS:06}.

Let $\Lambda(Q)$ be the set of vectors $d \in \smash{\Z^{Q_0}}$ such that $d_i = 0$ for all but finitely many $i \in Q_0$.
For $d \in \Lambda_+(Q) := \Lambda(Q) \cap \smash{\Z_{\geq 0}^{Q_0}}$ we fix $k$-vector spaces $V_i$ of dimension $d_i$ and we consider the finite-dimensional vector space
\begin{displaymath}
R(Q,d) = \bigoplus_{\alpha \in Q_1} \Hom(V_{s(\alpha)},V_{t(\alpha)}).
\end{displaymath}
Its elements are representations of $Q$ of dimension vector $d$. On $R(Q,d)$ we have an action of the group $G_d = \prod_{i \in Q_0} \GL_{d_i}(k)$ by
\begin{displaymath}
g \cdot M = \big(g_{t(\alpha)}M_\alpha g_{s(\alpha)}^{-1}\big)_{\alpha \in Q_1}.
\end{displaymath}
Two elements of $R(Q,d)$, considered as representations of $Q$, are isomorphic if and only if they lie in the same $G_d$-orbit. Note that the image $\Delta \sub G_d$ of the diagonal embedding of $\G_m$ acts trivially on $R(Q,d)$ so the action of $G_d$ descends to an action of $PG_d = G_d/\Delta$.

A $\Z$-linear map $\theta\colon \Lambda(Q) \to \Z$ is called a stability condition. For a vector $d \in \Lambda_+(Q) - \{0\}$ we define the slope of $d$ as
\begin{displaymath}
\mu(d) = \frac{\theta(d)}{\sum_i d_i}.
\end{displaymath}
Define $\mu(M) := \mu(\dimvect M)$.
A representation $M$ of $Q$ is called $\theta$-semi-stable ($\theta$-stable) if $\mu(M') \leq \mu(M)$ (resp.\ $\mu(M') < \mu(M)$) for every non-zero proper subrepresentation~$M'$ of~$M$.

We regard the vector space $R(Q,d)$ as a variety and $G_d$ as a linear algebraic group. The sets $R(Q,d)^{\theta-\sst}$ and $R(Q,d)^{\theta-\st}$ of semi-stable resp.\ stable points of $R(Q,d)$ are Zariski open (but they might be empty). Obviously
\begin{displaymath}
R(Q,d)^{\theta-\st} \sub R(Q,d)^{\theta-\sst} \sub R(Q,d).
\end{displaymath}
King shows in \cite[Proposition~3.1]{King:94} that the set of semi-stable points of $R(Q,d)$ agrees with the set of semi-stable points with respect to the $G_d$-linearization of the (trivial) line bundle on $R(Q,d)$ which is given by the character
\begin{displaymath}
\chi_\theta(g) = \prod_{i \in Q_0} \det(g_i)^{\theta(d)-\theta_i\sum_jd_j}.
\end{displaymath}
Note that $\chi_\theta$ is trivial on $\Delta$ whence it descends to a character of $PG_d$. The twist in the exponent in the definition of the character $\chi_\theta$ was introduced in \cite[Section~3.4]{Reineke:08} to get rid of the requirement $\theta(d) = 0$. The set $R(Q,d)^{\theta-\st}$ is the set of properly stable points (in the sense of Mumford \cite[Defenition~1.8]{GIT:94}) with respect to the aforementioned linearized line bundle. This holds as the isotropy group of a stable representation is $\Delta$ by Schur's lemma.

We define $M^{\theta-\sst}(Q,d) = R(Q,d)^{\theta-\sst}/\!\!/PG_d$ and $M^{\theta-\st}(Q,d) = R(Q,d)^{\theta-\st}/PG_d$. The quotient map $R(Q,d)^{\theta-\st} \to M^{\theta-\st}(Q,d)$ is a $PG_d$-torsor in the \'{e}tale topology (see \cite[Proposition~0.9]{GIT:94}) with a smooth total space, so we conclude that $M^{\theta-\st}(Q,d)$ is smooth. The induced morphism $M^{\theta-\sst}(Q,d) \to M^{0-\sst}(Q,d) = R(Q,d)/\!\!/PG_d$ is projective. A result of Le Bruyn and Procesi \cite[Theorem~1]{LP:90} states that the ring of invariants is generated by traces along oriented cycles in $Q$. If $Q$ has no oriented cycles, in which case we call $Q$ acyclic, then $M^{\theta-\sst}(Q,d)$ is projective. If $R(Q,d)^{\theta-\sst} = R(Q,d)^{\theta-\st}$ we call $\theta$ generic for $d$ and we write $R(Q,d)^\theta$ for the (semi\nobreakdash-)stable locus and $M^\theta(Q,d)$ for the quotient.

Let $\VV_i$ be the (trivial) vector bundle on $R(Q,d)$ with fiber $V_i$ equipped with the $G_d$-lineari\-za\-tion given by
\begin{displaymath}
g(M,v) = (g\cdot M,g_iv).
\end{displaymath}
Note that $\Delta$ does not act trivially on the fibers and hence $\VV_i$ does not descend along the geometric quotient $R(Q,d)^{\theta-\st} \to M^{\theta-\st}(Q,d)$. However, the bundles $\VV_i^\vee \otimes \VV_j$ do.

\section{Torus actions on quiver moduli}

Fix a finite quiver $Q$ and a dimension vector $d$. Let $T = \mathbb{G}_m^{Q_1}$ act on $R(Q,d)$ as follows: an element $t = (t_\alpha)_{\alpha} \in T$ acts on $M = (M_\alpha)_\alpha \in R(Q,d)$ by $t. M = (t_\alpha M_\alpha)_\alpha$. As it commutes with the $PG_d$-action on $R(Q,d)$, the action of $T$ descends to an action on the geometric quotient $M^{\theta-\st}(Q,d)$.

The locus of fixed points $M^{\theta-\st}(Q,d)^T$ can be described in terms of the universal abelian covering quiver $\smash{\hat{Q}}$. The (infinite) quiver $\smash{\hat{Q}}$ is given by $\smash{\hat{Q}_0} = Q_0 \times \Z^{Q_1}$, $\smash{\hat{Q}_1} = Q_1 \times \Z^{Q_1}$ and for $\alpha \in Q_1$ and $\chi \in \Z^{Q_1}$ the source and target of the arrow $(\alpha,\chi)$ of $\smash{\hat{Q}}$ are
\begin{gather*}
s(\alpha,\chi)  = (s(\alpha),\chi),\qquad  t(\alpha,\chi)  = (t(\alpha),\chi+x_\alpha).
\end{gather*}
The character $x_\alpha \in X(T)$ in the right-hand expression above is defined as $x_\alpha(t) = t_\alpha$.

We say that a dimension vector $\beta \in \smash{\Lambda_+\big(\hat{Q}\big)}$ covers $d$ if $\smash{\sum_\chi} \beta_{i,\chi} = d_i$ for every vertex $i$. There is an action of $\Z^{Q_1}$ on $\smash{\Lambda_+\big(\hat{Q}\big)}$ given by $\xi \cdot \beta = (\beta_{i,\chi+\xi})_{(i,\chi)}$. Two dimension vectors in the same orbit will be called translates of one another. Translates clearly cover the same dimension vector of $Q$. We define a stability condition $\smash{\hat{\theta}}$ for $\hat{Q}$ by $\smash{\hat{\theta}}_{(i,\chi)} = \theta_i$. The following result is due to Weist:

\begin{thm}[{\cite[Theorem~3.8]{Weist:13}}] \label{t:weist}
The fixed point locus $M^{\theta-\st}(Q,d)^T$ decomposes as
\begin{displaymath}
M^{\theta-\st}(Q,d)^T = \bigsqcup_{\beta} X_\beta,
\end{displaymath}
a disjoint union of irreducible components $X_\beta$, where the union ranges over all $\beta \in \Lambda_+\big(\hat{Q}\big)$ up to translation which cover $d$. Each $X_\beta$ is isomorphic to
\begin{displaymath}
X_\beta \cong M^{\hat{\theta}-\st}\big(\hat{Q},\beta\big).
\end{displaymath}
\end{thm}

\section{Torus-equivariant tautological relations}

Again let $Q$ be a finite quiver, let $d$ be a dimension vector for $Q$ and let $\theta$ be a stability condition. We compute a presentation for the rational equivariant Chow ring $A_T^*\big(M^{\theta-\st}(Q,d)\big)_\Q$. For the definition and general results on equivariant Chow rings we refer to~\cite{EG:98}.
We will work with Chow rings with rational coefficients throughout.

The $T$-equivariant Chow ring $A_T^*\big(M^{\theta-\st}(Q,d)\big)_\Q$ coincides with the $PG_d \times T$-equivariant Chow ring $A_{PG_d \times T}^*(R(Q,d)^{\theta-\st})_\Q$ by \cite[Lemma~2.1]{MRV:06}. This is not far away from the $G_d \times T$-equivariant Chow ring. So we would first like to compute the rings
\begin{gather*}
A_{G_d \times T}^*\big(R(Q,d)^{\theta-\st}\big)_\Q,\qquad A_{G_d \times T}^*\big(R(Q,d)^{\theta-\sst}\big)_\Q.
\end{gather*}
We will derive a presentation of these rings.

Choose a basis for each of the $V_i$. Let $T_d$ be the maximal torus of $G_d$ of diagonal matrices. Let us recap the definitions of the groups that we are using:
\begin{gather*}
G_d = \prod_{i \in Q_0} \GL(V_i) \cong \prod_{i \in Q_0} \GL_{d_i}(k), \\
T_d \sub G_d \text{ maximal torus of invertible diagonal matrices}, \\
\Delta = \{ (z\id_{V_i})_i \,|\,  z \in k^\times \} \sub G_d, \\
PG_d = G_d/\Delta, \\
PT_d = T_d/\Delta, \\
T = \G_m^{Q_1}.
\end{gather*}

Let $\xi_{i,r} \in X(T_d)$ be the character that selects in the matrix corresponding to the vertex $i$ the $r$\textsuperscript{th} diagonal entry. Recall that $x_\alpha \in X(T)$ is the character that selects the entry which corresponds to $\alpha$. For a character $\chi \in X(T_d \times T) = X(T_d) \oplus X(T)$ let $L(\chi)$ be the trivial line bundle on $R(Q,d)$ equipped with the $T_d \times T$-lineariz\-ation induced by $\chi$. So $\smash{c_1^{T_d \times T}(L(\chi))} = \chi$.
As $G_d \times T$ acts linearly on the vector space $R(Q,d)$, we obtain
\begin{gather*}
A_{G_d \times T}^*(R(Q,d))_\Q  \cong A_{G_d \times T}^*(\pt)_\Q
 \cong \bigg(\bigotimes_{i \in Q_0} \Q[\xi_{i,1},\dots,\xi_{i,d_i}]^{S_{d_i}}\bigg) \otimes_\Q \Q[x_\alpha]_{\alpha \in Q_1} \\
\hphantom{A_{G_d \times T}^*(R(Q,d))_\Q}{} = \bigg(\bigotimes_{i \in Q_0} \Q[x_{i,1},\dots,x_{i,d_i}]\bigg) \otimes_\Q \Q[x_\alpha]_{\alpha \in Q_1}.
\end{gather*}
In the above equation $x_{i,r}$ is the $r$\textsuperscript{th} elementary symmetric function in the variables $\xi_{i,1},\dots,\xi_{i,d_i}$. Note that $x_{i,r} = c_r^{G_d \times T}(\VV_i)$.
Let $j_1\colon R(Q,d)^{\theta-\sst} \to R(Q,d)$ and $j_2\colon R(Q,d)^{\theta-\st} \to R(Q,d)$ be the open embeddings.
The pull-backs of $j_1$ and $j_2$ induce surjective homomorphisms of graded rings
\begin{displaymath}
\begin{tikzcd}
A_{G_d \times T}^*(R(Q,d))_\Q \arrow{r}{j_1^*} \arrow{rd}{j_2^*}& A_{G_d \times T}^*\big(R(Q,d)^{\theta-\sst}\big)_\Q \arrow{d}{} \\
& A_{G_d \times T}^*\big(R(Q,d)^{\theta-\st}\big)_\Q.
\end{tikzcd}
\end{displaymath}
Let $i_1$ and $i_2$ be the closed embeddings of $R(Q,d) - R(Q,d)^{\theta-\sst}$ and $R(Q,d) - R(Q,d)^{\theta-\st}$, respectively.
The push-forwards of $i_1$ and $i_2$ give surjections onto the kernels of $j_1^*$ resp.\ $j_2^*$. The same arguments as in \cite[Theorems~8.1 and~9.1]{FR:18:Sst_ChowHa} show that the images of $i_1^*$ and $i_2^*$ can be re-written in terms of the $T$-equivariant CoHA multiplication. More precisely, we consider the correspondence diagram
\begin{displaymath}
R(Q,d') \times R(Q,d'') \ot \begin{pmatrix} R(Q,d') & * \\ 0 & R(Q,d'') \end{pmatrix} \to R(Q,d),
\end{displaymath}
where $d'$ and $d''$ are dimension vectors of $Q$ such that $d'+d''=d$.
The left-hand map is the projection and the right-hand map is the inclusion as a linear subspace. These maps are equivariant with respect to
\begin{displaymath}
G_{d'} \times G_{d''} \times T \ot \begin{pmatrix} G_{d'} & * \\ 0 & G_{d''} \end{pmatrix} \times T \to G_d \times T.
\end{displaymath}
Passing to equivariant Chow groups we obtain a map
\begin{displaymath}
A_{G_{d'} \times T}^*(R(Q,d'))_\Q \otimes_\Q A_{G_{d''} \times T}^*(R(Q,d''))_\Q \to A_{G_d \times T}^*(R(Q,d))_\Q,
\end{displaymath}
which sends $f \otimes g$ to $f * g$ which is given by
\begin{displaymath}
f*g = \sum_\pi f(\xi_{i,\pi^{i}(r)},x_\alpha)_{r=1,\dots,d'_i}\cdot g(\xi_{i,\pi^{i}(r)},x_\alpha)_{r=d'_i+1,\dots,d_i}\cdot\frac{\Delta_1(\xi_{i,\pi^{i}(r)},x_\alpha)}{\Delta_0(\xi_{i,\pi^{i}(r)})}.
\end{displaymath}
In the above formula $\pi = (\pi^{i})_{i \in Q_0}$ ranges over all elements of $\prod_{i \in Q_0}S_{d_i}$ for which each $\pi^{i}$ is a~$(d'_i,d''_i)$-shuffle permutation and
\begin{gather*}
\Delta_0  = \prod_{i \in Q_0} \prod_{r=1}^{d'_i} \prod_{s=d'_i+1}^{d_i} (\xi_{i,s}-\xi_{i,r}), \\
\Delta_1  = \prod_{\alpha:i \to j} \prod_{r=1}^{d'_i} \prod_{s=d'_j+1}^{d_j} (\xi_{j,s}-\xi_{i,r} + x_\alpha).
\end{gather*}
Recall that being a $(d'_i,d''_i)$-shuffle permutation means $\pi^i(1) < \dots < \pi^i(d'_i)$ and $\pi^i(d'_i+1) < \dots < \pi^i(d_i)$.
The contribution of $x_\alpha$ in the expression for $\Delta_1$ can be explained as follows: inside $\smash{\big(\begin{smallmatrix} R(Q,d') & * \\ 0 & R(Q,d'') \end{smallmatrix}\big)}$, the subset $R(Q,d') \times R(Q,d'')$ is the zero locus of a section of a $T_d \times T$-equivariant vector bundle. The bundle is isomorphic to
\begin{displaymath}
\bigoplus_{\alpha\colon i \to j} \bigoplus_{r=1}^{d'_i} \bigoplus_{s=d'_j+1}^{d_j} L(\xi_{i,r})^\vee \otimes L(\xi_{j,s}) \otimes L(x_\alpha).
\end{displaymath}

\begin{thm} \label{t:taut}The kernels of the surjections
\begin{displaymath}
\begin{tikzcd}
\big(\!\bigotimes_{i \in Q_0}\! \Q[\xi_{i,1},\dots,\xi_{i,d_i}]^{S_{d_i}}\!\big)\! \otimes_\Q \! \Q[x_\alpha]_{\alpha \in Q_1}\! = &[-3em] A_{G_d \times T}^*(R(Q,d))_\Q\! \arrow{r}{j_1^*} \arrow{rd}{j_2^*} & \!A_{G_d \times T}^*\!\big(R(Q,d)^{\theta-\sst}\big)_\Q \arrow{d}{} \\
&& A_{G_d \times T}^*\!\big(R(Q,d)^{\theta-\st}\big)_\Q
\end{tikzcd}
\end{displaymath}
are the $\Q$-linear subspaces
\begin{gather*}
\ker(j_1^*)  = \sum_{\substack{d',d'' \in \smash{\Z_{\geq 0}^{Q_0}}\\ d'+d''=d\\ \mu(d') > \mu(d'')}} \big\{ f * g \,|\,  f \in A_{G_{d'} \times T}^*(R(Q,d'))_\Q, g \in A_{G_{d''} \times T}^*(R(Q,d''))_\Q \big\}, \\
\ker(j_2^*)  = \sum_{\substack{d',d'' \in \smash{\Z_{\geq 0}^{Q_0}}\\ d'+d''=d\\ \mu(d') \geq \mu(d'')}} \big\{ f * g \,|\,  f \in A_{G_{d'} \times T}^*(R(Q,d'))_\Q, g \in A_{G_{d''} \times T}^*(R(Q,d''))_\Q \big\}.
\end{gather*}
\end{thm}

\begin{proof}The statement of the theorem follows from a suitable adaption of the arguments in \cite[Theorems~8.1 and~9.1]{FR:18:Sst_ChowHa}.
\end{proof}

Now to the $T$-equivariant Chow ring of $M^{\theta-\st}(Q,d)$. The maximal torus $T_d$ of $G_d$ con\-tains~$\Delta$. The quotient $PT_d := T_d/\Delta$ is a maximal torus of $PG_d$. Its character lattice is given by
\begin{displaymath}
X(PT_d) = \left\{ \eta = \sum_{i \in Q_0} \sum_{r=1}^{d_i} b_{i,r} \xi_{i,r} \,|\,  \sum_{i \in Q_0} \sum_{r=1}^{d_i} b_{i,r} = 0 \right\} \sub X(T_d).
\end{displaymath}
The Weyl group of $PT_d$ inside $PG_d$ is also $W_d := \prod_{i \in Q_0} S_{d_i}$. Therefore
\begin{displaymath}
A_{PG_d \times T}^*(R(Q,d))_\Q \cong S(X(PT_d))_\Q^{W_d} \otimes_\Q \Q[x_\alpha]_{\alpha \in Q_1}
\end{displaymath}
by \cite[Proposition~6]{EG:98}.
The ring homomorphism $\smash{A_{PG_d \times T}^*}(R(Q,d))_\Q \to \smash{A_{G_d \times T}^*}(R(Q,d))_\Q$ which comes from $G_d \to PG_d$ corresponds to the homomorphism that is induced by the inclusion $X(PT_d) \into X(T_d)$. We can give generators for $\smash{S(X(PT_d))_\Q^{W_d}}$. Fix an order on $Q_0$, say $Q_0 = \{1,\dots,n\}$. Then consider the lattice $Z_d = \smash{\bigoplus_{1 \leq i < j \leq n}} \Z^{d_i} \otimes \Z^{d_j}$. Let $\smash{\zeta_{r,s}^{i,j}}$ be the pure tensor of unit vectors $e_r \otimes e_s$ embedded into the $(i,j)$\textsuperscript{th} direct summand of $Z_d$. On $Z_d$ there is an action of $W_d$ in the obvious way. The map $Z_d \to X(PT_d)$ which sends $\smash{\zeta_{r,t}^{i,j}}$ to $\xi_{i,r} - \xi_{j,t}$ is well-defined, $W_d$-equivariant, and surjective. It hence gives rise to a surjective homomorphim $f\colon \smash{S(Z_d)_\Q^{W_d}} \to \smash{S(X(PT_d))_\Q^{W_d}}$. The ring $S(Z_d)_\Q^{W_d}$ is generated by the algebraically independent elements
\begin{displaymath}
z_{k,l}^{i,j} := \sum_{\substack{1 \leq r_1 < \dots < r_k \leq d_i \\ 1 \leq t_1 < \dots < t_l \leq d_j}} \prod_{\nu=1}^k \prod_{\mu=1}^l \zeta_{r_\nu,t_\mu}.
\end{displaymath}

\begin{rem}Suppose that $k = \C$. The same arguments as in \cite[Theorem~5.1]{FR:18:Sst_ChowHa} show that the $G_d \times T$-equivariant and also the $PG_d \times T$-equivariant cohomology of $R(Q,d)^{\theta-\sst}$ with rational coefficients vanishes in odd degrees and that the cycle maps $\smash{A_{G_d \times T}^i}\big(R(Q,d)^{\theta-\sst}\big)_\Q \to \smash{H_{G_d \times T}^{2i}}\big(R(Q,d)^{\theta-\sst};\Q\big)$ and $\smash{A_{PG_d \times T}^i}\big(R(Q,d)^{\theta-\sst}\big)_\Q \to \smash{H_{PG_d \times T}^{2i}}\big(R(Q,d)^{\theta-\sst};\Q\big)$ are isomorphisms. However, if semi-stability and stability do not agree then the equivariant cohomology groups of $R(Q,d)^{\theta-\st}$ will in general not vanish in odd degrees and the cycle maps will not be isomorphisms.
\end{rem}

\section{Localization at torus fixed points}

In this section we compute the pull-back $\iota^*\colon A_T^*\big(M^{\theta-\st}(Q,d)\big)_\Q \to A_T^*\big(M^{\theta-\st}(Q,d)^T\big)_\Q$ along the inclusion $\iota\colon M^{\theta-\st}(Q,d)^T \to M^{\theta-\st}(Q,d)$.
By Weist's result the locus of torus fixed points $M^\theta(Q,d)$ decomposes into components which are isomorphic to $\smash{M^{\hat{\theta}}\big(\hat{Q},\beta\big)}$. It therefore suffices to compute the pull-back
\begin{displaymath}
\iota_\beta^*\colon \ A_T^*\big(M^{\theta-\st}(Q,d)\big)_\Q \to A_T^*\big(M^{\hat{\theta}-\st}\big(\hat{Q},\beta\big)\big)_\Q
\end{displaymath}
along the regular closed immersion
\begin{displaymath}
\iota_\beta\colon \ M^{\hat{\theta}}\big(\hat{Q},\beta\big) \to M^\theta(Q,d).
\end{displaymath}
Fix for every $i \in Q_0$ a decomposition $V_i = \bigoplus_{\chi \in X(T)} V_{i,\chi}$ into subspaces of dimension $\dim V_{i,\chi} = \beta_{i,\chi}$. This amounts to embedding $G_\beta$ as a Levi subgroup of $G_d$. The immersion $\iota_\beta$ is provided by the map
\begin{displaymath}
\tilde{\iota}_{\beta}\colon \ R\big(\hat{Q},\beta\big) \to R(Q,d)
\end{displaymath}
which sends a representation $N \in R\big(\hat{Q},\beta\big)$ to $M = \tilde{\iota}_\beta(N) \in R(Q,d)$ where $M_\alpha\colon V_i \to V_j$ is defined by
\begin{displaymath}
M_\alpha\bigg(\sum_\chi v_\chi\bigg) = \sum_\chi N_{\alpha,\chi}(v_\chi)
\end{displaymath}
for $v_{i,\chi} \in V_{i,\chi}$. The map $\tilde{\iota}_\beta$ is $G_\beta \times T$-equivariant with respect to the $G_\beta \times T$-action $a_1$ on $R(\hat{Q},\beta)$ which is defined by
\begin{displaymath}
a_1\colon \ (G_\beta \times T) \times R\big(\hat{Q},\beta\big) \to R\big(\hat{Q},\beta\big),\qquad ((g,t), N) \mapsto \big(t_\alpha g_{j,\chi+x_\alpha}N_{\alpha,\chi}g_{i,\chi}^{-1}\big)_{\alpha,\chi}.
\end{displaymath}
With respect to this action, $R\big(\hat{Q},\beta\big)^{\hat{\theta}-\st}$ is invariant and the action descends to an action of $PG_\beta \times T$.
To compute the pull-back of $\tilde{\iota}_\beta$ in equivariant intersection theory we choose a basis $e_{i,\chi,1},\dots,e_{i,\chi,\beta_{i,\chi}}$ of $V_{i,\chi}$ and a bijection
\begin{align*}
\phi_i^{(\beta)}\colon \ \{1,\dots,d_i\} &\to \{ (\chi,s) \,|\,  \chi \in X(T),\ s \in \{1,\dots,\beta_{i,\chi} \} \}, \\
r &\mapsto \big(\chi^{(\beta)}_{i,r},s^{(\beta)}_{i,r}\big).
\end{align*}
For convenience, we are going to neglect the dependency on $\beta$ in the notation whenever possible. Let $e_{i,r} := e_{i,\chi_{i,r},s_{i,r}}$. Then $e_{i,1},\dots,e_{i,d_i}$ is a basis of~$V_i$.
Consider the maximal torus $T_d$ of $G_d$ of diagonal matrices with respect to that basis; it is contained in the Levi subgroup~$G_\beta$. Its character lattice is
\begin{gather*}
\bigoplus_{i \in Q_0} \bigoplus_{r=1}^{d_i} \Z\xi_{i,r} = \bigoplus_{i \in Q_0} \bigoplus_{\chi \in X(T)} \bigoplus_{s=1}^{\beta_{i,\chi}} \Z\xi_{i,\chi,s},
\end{gather*}
where $\xi_{i,r} := \xi_{i,\chi_{i,r},s_{i,r}}$. Now
\begin{gather*}
A_{G_d \times T}^*(R(Q,d))_\Q  = \bigg(\bigotimes_{i \in Q_0} \Q[\xi_{i,1},\dots,\xi_{i,d_i}]^{S_{d_i}} \bigg) \otimes_\Q \Q[x_\alpha]_{\alpha \in Q_1} \\
\hphantom{A_{G_d \times T}^*(R(Q,d))_\Q}{}
= \bigg(\bigotimes_{i \in Q_0} \Q[x_{i,1},\dots,x_{i,d_i}] \bigg) \otimes_\Q \Q[x_\alpha]_{\alpha \in Q_1}, \\
A_{G_\beta \times T}^*(R(\hat{Q},\beta)) _\Q  = \bigg( \bigotimes_{i \in Q_0} \bigotimes_{\chi \in X(T)} \Q[\xi_{i,\chi,1},\dots,\xi_{i,\chi,\beta_{i,\chi}}]^{S_{\beta_{i,\chi}}} \bigg)\otimes_\Q \Q[x_\alpha]_{\alpha \in Q_1} \\
\hphantom{A_{G_\beta \times T}^*(R(\hat{Q},\beta)) _\Q}{}
= \bigg( \bigotimes_{i \in Q_0} \bigotimes_{\chi \in X(T)} \Q[x_{i,\chi,1},\dots,x_{i,\chi,\beta_{i,\chi}}] \bigg)\otimes_\Q \Q[x_\alpha]_{\alpha \in Q_1},
\end{gather*}
where $x_{i,r} = e_r(\xi_{i,1},\dots,\xi_{i,d_i})$ and $x_{i,\chi,s} = e_s(\xi_{i,\chi,1},\dots,\xi_{i,\chi,\beta_{i,\chi}})$. The map $\tilde{\iota}_\beta^*$ is given by $\tilde{\iota}_\beta^*(x_\alpha) = x_\alpha$ and by $\tilde{\iota}_\beta^*(\xi_{i,r}) = \xi_{i,\chi_{i,r},s_{i,r}}$. This implies
\begin{displaymath}
\tilde{\iota}_\beta^*(x_{i,r}) = e_r(\xi_{i,\chi_1,s_1},\dots,\xi_{i,\chi_{d_i},s_{d_i}}).
\end{displaymath}
In principle the image of $x_{i,r}$ under $\smash{\tilde{\iota}_\beta^*}$ can also be expressed in terms of the $\smash{x_{i,\chi,s}}$'s, like $\smash{\tilde{\iota}_\beta^*(x_{i,1})} = \smash{\sum_\chi x_{i,\chi,1}}$ and $\smash{\tilde{\iota}_\beta^*(x_{i,d_i})} = \smash{\prod_\chi x_{i,\chi,\beta_{i,\chi}}}$, but the intermediate terms are more complicated.

We obtain the following commutative diagram:
\begin{displaymath}
\begin{tikzcd}[column sep=small]
&[-2em] A_T^*\big(M^{\theta-\st}(Q,d)\big)_\Q \arrow{rr}{\iota_\beta^*} \arrow{d}{\cong} &[-2em] &[-2em]
A_T^*\big(M^{\hat{\theta}-\st}(\hat{Q},\beta)\big)_\Q \arrow{d}{\cong} \\
& A_{PG_d \times T}^*\big(R(Q,d)^{\theta-\st}\big)_\Q \arrow{rr}{} \arrow{dl}{} \arrow[twoheadleftarrow]{dd}{} & &
A_{PG_\beta \times T}^*\big(R(\hat{Q},\beta)^{\hat{\theta}-\st}\big)_\Q \arrow{dl}{} \\
A_{G_d \times T}^*\big(R(Q,d)^{\theta-\st}\big)_\Q \arrow[crossing over]{rr}{} & &
A_{G_\beta \times T}^*\big(R(\hat{Q},\beta)^{\hat{\theta}-\st}\big)_\Q & \\
& A_{PG_d \times T}^*(R(Q,d))_\Q \arrow[pos=.25]{rr}{\tilde{\iota}_\beta^*} \arrow[hook]{dl}{} & &
A_{PG_\beta \times T}^*(R(\hat{Q},\beta))_\Q \arrow[hook]{dl}{} \arrow[two heads]{uu}{j_\beta^*} \\
A_{G_d \times T}^*(R(Q,d))_\Q \arrow{rr}{\tilde{\iota}_\beta^*} \arrow[two heads]{uu}{} & &
A_{G_\beta \times T}^*(R(\hat{Q},\beta))_\Q \arrow[two heads,crossing over,pos=.25]{uu}{j_\beta^*} &
\end{tikzcd}
\end{displaymath}
where $j_\beta\colon R\big(\hat{Q},\beta\big)^{\hat{\theta}-\st} \to R\big(\hat{Q},\beta\big)$ is the open embedding of the stable locus. In the above diagram the action of $G_\beta \times T$ and $PG_\beta \times T$ is by~$a_1$.
As $T$ acts trivially on $\smash{M^{\hat{\theta}-\st}\big(\hat{Q},\beta\big)}$, we may identify
\begin{gather*}
A_T^*\big(M^{\hat{\theta}-\st}(\hat{Q},\beta)\big)_\Q \cong A^*\big(M^{\hat{\theta}-\st}\big(\hat{Q},\beta\big)\big)_\Q \otimes_\Q S(X(T))_\Q \\
\hphantom{A_T^*\big(M^{\hat{\theta}-\st}(\hat{Q},\beta)\big)_\Q}{} \cong A_{PG_\beta}^*\big(R\big(\hat{Q},\beta\big)^{\hat{\theta}-\st}\big)_\Q \otimes_\Q S(X(T))_\Q.
\end{gather*}
We introduce another action $a_2$ of $G_\beta \times T$ on $\smash{R\big(\hat{Q},\beta\big)}$. Define
\begin{displaymath}
a_2\colon \ (G_\beta \times T) \times R\big(\hat{Q},\beta\big) \to R\big(\hat{Q},\beta\big),\qquad ((g,t), N) \mapsto \big(g_{j,\chi+x_\alpha}N_{\alpha,\chi}g_{i,\chi}^{-1}\big)_{\alpha,\chi}.
\end{displaymath}
The action $a_2$ leaves the $\smash{\hat{\theta}}$-stable locus invariant and descends to an action of $PG_\beta \times T$. With respect to the action $a_2$ we obtain
\begin{displaymath}
A_{PG_\beta}^*\big(R\big(\hat{Q},\beta\big)^{\hat{\theta}-\st}\big)_\Q \otimes_\Q S(X(T))_\Q \cong A_{PG_\beta \times T}^*\big(R\big(\hat{Q},\beta\big)^{\hat{\theta}-\st}\big)_\Q.
\end{displaymath}

We are going to provide an isomorphism between the equivariant Chow rings with respect to the actions $a_1$ and $a_2$.
To this end we consider the automorphism of the group $G_\beta \times T$ given by
\begin{displaymath}
\Phi\colon \ G_\beta \times T \to G_\beta \times T,\qquad (g,t) \mapsto ((\chi(t)g_{i,\chi})_{i,\chi},t).
\end{displaymath}
Then we get a commutative diagram
\begin{displaymath}
\begin{tikzcd}
(G_\beta \times T) \times R\big(\hat{Q},\beta\big) \arrow{r}{\Phi \times \id} \arrow{d}{a_1} & (G_\beta \times T) \times R\big(\hat{Q},\beta\big) \arrow{d}{a_2} \\
R\big(\hat{Q},\beta\big) \arrow{r}{\id} & R\big(\hat{Q},\beta\big).
\end{tikzcd}
\end{displaymath}
This yields an isomorphism of quotient stacks $\phi\colon  \big[R\big(\hat{Q},\beta\big)/(G_\beta \times T)\big]_{\text{via } a_1} \to \big[R\big(\hat{Q},\beta\big)/(G_\beta \times T)\big]_{\text{via }a_2}$. We regard the equivariant Chow ring as the Chow ring of the corresponding quotient stack. This is justified by arguments by Edidin and Graham \cite[Section~5.3]{EG:98} and by Kresch \cite{Kresch:99}. The pull-back $\phi^*$ in intersection theory sends $\phi^*(x_\alpha) = x_\alpha$ and $\phi^*(\xi_{i,\chi,s}) = \xi_{i,\chi,s}+\chi$. We get another commutative diagram:{\footnotesize
\begin{displaymath}
\begin{tikzcd}[column sep=small]
&[-5em] A_T^*\big(M^{\hat{\theta}-\st}\big(\hat{Q},\beta\big)\big)_\Q \arrow{rrrr}{\cong} \arrow{d}{\cong} &[-6em] &[-5em]
&[-6em] &[-5em]
A^*\big(M^{\hat{\theta}-\st}\big(\hat{Q},\beta\big)\big)_\Q \otimes_\Q S \arrow{d}{\cong} \\[1em]
& A_{PG_\beta \times T}^*\big(R\big(\hat{Q},\beta\big)^{\hat{\theta}-\st}\big)_\Q \arrow{rr}{\cong} \arrow{dl}{} \arrow[twoheadleftarrow,pos=.75]{dd}{j_\beta^*} & &
A_{PG_\beta \times T}^*\big(R\big(\hat{Q},\beta\big)^{\hat{\theta}-\st}\big)_\Q \arrow{rr}{\cong} \arrow{dl}{} \arrow[twoheadleftarrow,pos=.75]{dd}{j_\beta^*} & &
A_{PG_\beta}^*\big(R\big(\hat{Q},\beta\big)^{\hat{\theta}-\st}\big)_\Q \otimes_\Q S \arrow{dl}{} \arrow[twoheadleftarrow]{dd}{j_\beta^* \otimes \id} \\
A_{G_\beta \times T}^*\big(R\big(\hat{Q},\beta\big)^{\hat{\theta}-\st}\big)_\Q\arrow[crossing over,pos=.25]{rr}{\cong} & &
A_{G_\beta \times T}^*\big(R\big(\hat{Q},\beta\big)^{\hat{\theta}-\st}\big)_\Q \arrow[crossing over,pos=.25]{rr}{\cong} & &
A_{G_\beta}^*\big(R\big(\hat{Q},\beta\big)^{\hat{\theta}-\st}\big)_\Q \otimes_\Q S & \\
& A_{PG_\beta \times T}^*\big(R\big(\hat{Q},\beta\big)\big)_\Q \arrow[pos=.85]{rr}{(\phi^{-1})^*}[swap]{\cong} \arrow[hook]{dl}{} & &
A_{PG_\beta \times T}^*\big(R\big(\hat{Q},\beta\big)\big)_\Q \arrow[hook]{dl}{} \arrow[pos=.75]{rr}{\cong} \arrow[hook]{dl}{} & &
A_{PG_\beta}^*\big(R\big(\hat{Q},\beta\big)\big)_\Q \otimes_\Q S \arrow[hook]{dl}{}\\
A_{G_\beta \times T}^*\big(R\big(\hat{Q},\beta\big)\big)_\Q \arrow{rr}{(\phi^{-1})^*}[swap]{\cong} \arrow[two heads]{uu}{j_\beta^*} & &
A_{G_\beta \times T}^*\big(R\big(\hat{Q},\beta\big)\big)_\Q \arrow[two heads,crossing over,pos=.25]{uu}{j_\beta^*} \arrow{rr}{\cong} & &
A_{G_\beta}^*\big(R\big(\hat{Q},\beta\big)\big)_\Q \otimes_\Q S \arrow[two heads,crossing over,pos=.25]{uu}{j_\beta^* \otimes \id} &
\end{tikzcd}
\end{displaymath}}

\noindent
Here we have used the short-hand $S := S(X(T))_\Q = \Q[x_\alpha]_{\alpha \in Q_1}$.
In the left-hand side wall of the diagram, the action is via $a_1$ while in the middle plane the action is via $a_2$. In the right-hand side wall of the diagram, all morphisms are of the form $r \otimes \id$ for some homomorphism $r$ between the respective Chow rings.

Let $i_\beta := \tilde{\iota}_\beta \circ \phi^{-1}$. It is a morphism of quotient stacks
\begin{gather*}
i_\beta\colon \ \big[R\big(\hat{Q},\beta\big)/(G_\beta \times T)\big]_{\text{via }a_2} \to [R(Q,d)/(G_d \times T)].
\end{gather*}
The pull-back along this morphism exists and is given as follows:

\begin{lem}Let $\beta$ be a dimension vector which covers $d$.
\begin{enumerate}\itemsep=0pt
\item The pull-back map in equivariant intersection theory of $i_\beta$ is the map
\begin{gather*}
i_\beta^*\colon \ \bigg(\bigotimes_{i \in Q_0} \Q[x_{i,1},\dots,x_{i,d_i}] \bigg) \otimes_\Q \Q[x_\alpha]_{\alpha \in Q_1} \\
\qquad {}\to \bigg( \bigotimes_{i \in Q_0} \bigotimes_{\chi \in X(T)} \Q[x_{i,\chi,1},\dots,x_{i,\chi,\beta_{i,\chi}}] \bigg)\otimes_\Q \Q[x_\alpha]_{\alpha \in Q_1}
\end{gather*}
defined by $i_\beta^*(x_{i,r}) = e_r(\xi_{i,\chi_1,s_1}-\chi_1,\dots,\xi_{i,\chi_{d_i},s_{d_i}}-\chi_{d_i})$ and by $i_\beta^*(x_\alpha) = x_\alpha$ where $\alpha \in Q_1$.

\item The pull-backs of the restrictions
\begin{gather*}
\big[R\big(\hat{Q},\beta\big)^{\hat{\theta}-\sst}/(G_\beta \times T)\big]_{\text{\rm via }a_2}  \to \big[R(Q,d)^{\theta-\sst}/(G_d \times T)\big], \\
\big[R\big(\hat{Q},\beta\big)^{\hat{\theta}-\st}/(G_\beta \times T)\big]_{\text{\rm via }a_2}  \to \big[R(Q,d)^{\theta-\st}/(G_d \times T)\big]
\end{gather*}
of $i_\beta$ are the induced maps by $i_\beta^*$ on the quotients by the ideals described in Theorem~{\rm \ref{t:taut}}.
\end{enumerate}
\end{lem}

Next we determine the image of the localization map
\begin{displaymath}
\iota^*\colon \ A_T^*(M^{\theta-\st}(Q,d))_\Q \to A^*(M^{\theta-\st}(Q,d)^T)_\Q \otimes_\Q S(X(T))_\Q.
\end{displaymath}
Recall that $\smash{S(Z_d)_\Q^{W_d}}$ surjects onto $\smash{A_{PG_d}^*(R(Q,d))_\Q}$. Consider the commutative diagram{\footnotesize
\begin{displaymath}
\begin{tikzcd}
S(Z_d)_\Q \arrow[two heads]{r}{} &[-1.5em] A_{PT_d \times T}^*(R(Q,d))_\Q \arrow[two heads]{r}{} &[-1.5em] A_{PT_d \times T}^*\big(R(Q,d)^{\theta-\st}\big)_\Q \arrow{r}{} &[-1.5em] A_{PT_\beta \times T}^*\big(R\big(\hat{Q},\beta\big)^{\hat{\theta}-\st}\big)_\Q \\[-.5em]
S(Z_d)_\Q^{W_d} \arrow[two heads]{r}{} \arrow[hook]{u}{} \arrow[two heads]{rd}{} & A_{PT_d \times T}^*(R(Q,d))_\Q^{W_d} \arrow[two heads]{r}{} \arrow[hook]{u}{} & A_{PT_d \times T}^*\big(R(Q,d)^{\theta-\st}\big)_\Q^{W_d} \arrow{r}{} \arrow[hook]{u}{} & A_{PT_\beta \times T}^*\big(R\big(\hat{Q},\beta\big)^{\hat{\theta}-\st}\big)_\Q^{W_\beta} \arrow[hook]{u}{} \\[-.5em]
& A_{PG_d \times T}^*(R(Q,d))_\Q \arrow[two heads]{rd}{} \arrow{u}{\cong} \arrow[two heads]{r}{} & A_{PG_d \times T}^*\big(R(Q,d)^{\theta-\st}\big)_\Q \arrow{u}{\cong} \arrow{r}{i_\beta^*} & A_{PG_\beta}^*\big(R\big(\hat{Q},\beta\big)^{\hat{\theta}-\st}\big)_\Q \otimes \Q[x_\alpha]_{\alpha \in Q_1} \arrow{u}{\cong} \\
& & A_{T}^*\big(M^{\theta-\st}(Q,d)\big)_\Q \arrow{r}{\iota_\beta^*} \arrow{u}{\cong} & A^*\big(M^{\hat{\theta}-\st}\big(\hat{Q},\beta\big)\big)_\Q \otimes_\Q \Q[x_\alpha]_{\alpha \in Q_1} \arrow{u}{\cong}
\end{tikzcd}
\end{displaymath}}

\noindent
In the right-most column, $PG_\beta \times T$ acts via $a_2$.
Let $f_\beta\colon S(Z_d)_\Q^{W_d} \to A^*\big(M^{\hat{\theta}-\st}\big(\hat{Q},\beta\big)\big)_\Q \otimes_\Q \Q[x_\alpha]_{\alpha \in Q_1}$ be the composition of two diagonal arrows with the arrow $\iota_\beta^*$. It maps
\begin{displaymath}
f_\beta(z_{k,l}^{i,j}) = \sum_{\substack{1 \leq r_1 < \dots < r_k \leq d_i \\ 1 \leq t_1 < \dots < t_l \leq d_j}} \prod_{\nu=1}^k \prod_{\mu=1}^l (\xi_{i,\chi_{i,r_\nu},s_{i,r_\nu}} - \chi_{i,r_\nu} - \xi_{j,\chi_{j,t_\mu},s_{j,t_\mu}} + \chi_{j,t_\mu}).
\end{displaymath}
Note that $\xi_{i,\chi,s} - \xi_{j,\eta,t}$ are the (non-equivariant) Chern roots of the bundle $\VV_{i,\chi} \otimes \VV_{j,\eta}^\vee$ on $\smash{M^{\hat{\theta}-\st}\big(\hat{Q},\beta\big)}$. Note also that $\smash{f_\beta\big(z_{k,l}^{i,j}\big)}$ is independent of the choice of a representative from the class of translates of $\beta$. This shows:

\begin{thm} \label{t:main}
The image of the pull-back map
\begin{gather*}
\iota^*\colon \ A_T^*\big(M^{\theta-\st}(Q,d)\big)_\Q \to A^*\big(M^{\theta-\st}(Q,d)^T\big)_\Q \otimes_\Q S(X(T))_\Q \\
\hphantom{\iota^*\colon \ A_T^*\big(M^{\theta-\st}(Q,d)\big)_\Q\to}{} = \bigg( \bigoplus_\beta A^*\big(M^{\hat{\theta}-\st}\big(\hat{Q},\beta\big)\big) \bigg)\otimes_\Q \Q[x_\alpha]_{\alpha \in Q_1}
\end{gather*}
of the inclusion of the fixed point locus is the subring which is generated by the elements
$\big(f_\beta\big(z_{k,l}^{i,j}\big)\big)_\beta$
where $i \neq j$ are vertices of $Q$ and $k=1,\dots,d_i$, $l = 1,\dots,d_j$ and by elements of the form $(1,\dots,1) \otimes x_\alpha$ with $\alpha \in Q_1$.
\end{thm}

The description of the image in Theorem~\ref{t:main} is hard to handle in general but the case of an action with isolated fixed points is more manageable. Assume that each of the covering roots $\beta$ of $d$ is a real root of $\smash{\hat{Q}}$. In this case, each of the fixed point components $\smash{M^{\hat{\theta}-\st}\big(\hat{Q},\beta\big)}$ is a single point. As there are only finitely many covering dimension vectors with connected support up to translation for $d$, this means that $T$ acts with finitely many fixed points. In this case
\begin{displaymath}
f_\beta\big(z_{k,l}^{i,j}\big) = \sum_{\substack{1 \leq r_1 < \dots < r_k \leq d_i \\ 1 \leq t_1 < \dots < t_l \leq d_j}} \prod_{\nu=1}^k \prod_{\mu=1}^l \big(\chi^{(\beta)}_{j,t_\mu} - \chi^{(\beta)}_{i,r_\nu}\big).
\end{displaymath}

\begin{cor} \label{c:isolated}
Suppose that every covering root of $d$ is a real root of $\hat{Q}$. Let $B = \{\beta_1,\dots,\beta_N\}$ be a set of representatives of translation classes of covering roots of $d$ for which there exists a~$\smash{\hat{\theta}}$-stable representation. Then the image of the pull-back map
\begin{displaymath}
\iota^*\colon \  A_T^*\big(M^{\theta-\st}(Q,d)\big)_\Q \to \Q[x_\alpha]_{\alpha \in Q_1}^{\oplus N}
\end{displaymath}
of the inclusion of the fixed point locus is the subring which is generated by the elements
\begin{displaymath}
\bigg( \sum_{\substack{1 \leq r_1 < \dots < r_k \leq d_i \\ 1 \leq t_1 < \dots < t_l \leq d_j}} \prod_{\nu=1}^k \prod_{\mu=1}^l \big(\chi^{(\beta_1)}_{j,t_\mu} - \chi^{(\beta_1)}_{i,r_\nu}\big),\dots,\sum_{\substack{1 \leq r_1 < \dots < r_k \leq d_i \\ 1 \leq t_1 < \dots < t_l \leq d_j}} \prod_{\nu=1}^k \prod_{\mu=1}^l \big(\chi^{(\beta_N)}_{j,t_\mu} - \chi^{(\beta_N)}_{i,r_\nu}\big) \bigg),
\end{displaymath}
where $i \neq j$ are vertices of $Q$ and $k=1,\dots,d_i$, $l = 1,\dots,d_j$ and by the elements $(x_\alpha,\dots,x_\alpha)$ for $\alpha \in Q_1$.
\end{cor}

\begin{ex} \label{e:proj_space}
Let us confirm the above formula in a well-known example. Let $Q$ be the $n+1$-Kronecker quiver. That is the quiver with vertices $i$ and $j$ and arrows $a_0,\dots,a_n\colon i \to j$. Let $d = (1,1)$ and let $\theta = (1,-1)$. A representation of $Q$ of dimension vector $d$ is a tuple $(p_0,\dots,p_n) \in k^n$. It is (semi\nobreakdash-)stable with respect to $\theta$ if and only if $(p_0,\dots,p_n) \neq (0,\dots,0)$. An isomorphism $M^\theta(Q,d) \to \P^n$ is provided by $(p_0,\dots,p_n) \mapsto [p_0:\dots:p_n]$.

The torus $T = \G_m^{n+1}$ acts on $\P^n$ by
\begin{displaymath}
t.p = [t_0p_0:\dots:t_np_n].
\end{displaymath}
Let $x_i$ be the character of $T$ given by $x_i(t) = t_i$. We equip the line bundle $\OO(-1)$ with a $T$-linearization as follows: for $v = (v_0,\dots,v_n)$ in the fiber of $\OO(-1)$ over a point $p$ let $t.v = (t_0v_0,\dots,t_nv_n)$. Then define
\begin{displaymath}
h := c_1^T(\OO(1)).
\end{displaymath}
Let $s_0,\dots,s_n \in H^0\big(\P^n,\OO(1)\big)$ be the usual sections. Then $s_\nu$ is of weight $-x_\nu$ with respect to the induced $T$-action on global sections. This implies that $s_\nu$ is a $T$-invariant section of the bundle $\OO(1) \otimes L(x_\nu)$, where $L(x_\nu)$ is the trivial bundle on $\P^n$ equipped with the $T$-linearization given by $x_\nu$. The bundle $\OO(1) \otimes (L(x_0) \oplus \dots \oplus L(x_n))$ hence has a no-where vanishing $T$-invariant global section which shows $(x_0+h)\cdots(x_n+h) = 0$ in $A_T^*\big(\P^n\big)_\Q$. It can be shown that this is the only relation, so
\begin{displaymath}
A_T^*(\P^n)_\Q = \Q[x_0,\dots,x_n,h]/(x_0+h)\cdots(x_n+h).
\end{displaymath}
The locus of $T$-fixed points of $\P^n$ is $\{e_0,\dots,e_n\}$, where $e_\nu = [0: \dots : 1 : \dots : 0]$ is the unit vector with an entry in the $\nu$\textsuperscript{th} position. Let $\iota_\nu$ be the inclusion of $\{e_\nu\}$ into $\P^n$. Then $\iota_\nu^*\OO(1) = L(-x_\nu)$, so $\iota_\nu^*(h) = -x_\nu$. The pull-back of the inclusion $\iota\colon (\P^n)^T \to \P^n$ is therefore the map
\begin{displaymath}
\iota^*\colon \ \Q[x_0,\dots,x_n,h]/(x_0+h)\cdots(x_n+h) \to \Q[x_0,\dots,x_n]^{\oplus n+1}
\end{displaymath}
given by $\iota^*(x_\nu) = (x_\nu,\dots,x_\nu)$ and $\iota^*(h) = -(x_0,\dots,x_n)$.

Let us try and match this with the formula from the previous corollary. The $n+1$ fixed points of the $T$-action on $M^\theta(Q,d)$ are given by dimension vectors $\beta_0,\dots,\beta_n$; the entries of these dimension vectors are $0$ or $1$ and the support of $\beta_\nu$ is the sub-quiver
\begin{displaymath}
(i,0) \xto{}{a_\nu} (j,x_\nu)
\end{displaymath}
of $\hat{Q}$. Corollary \ref{c:isolated} tells us that the image of $\iota^*$ is the subring of $\Q[x_0,\dots,x_n]^{\oplus n+1}$ which is generated by the elements $(x_\nu,\dots,x_\nu)$ and by $\smash{z_{1,1}^{i,j}} = (x_0,\dots,x_n)$. So the two descriptions of the image of $\iota^*$ do indeed agree.
\end{ex}

\begin{ex}
Let $Q$ be the $3$-Kronecker quiver. That is the quiver with two vertices $i$ and $j$ and 3 arrows $a,b,c\colon i \to j$. Consider the stability condition $\theta = (3,-2)$ and the dimension vector $d = (2,3)$ (so $\theta(d) = 0$). The dimension of the moduli space $M^\theta(Q,d)$ is $1 - \langle d,d \rangle_Q = 6$. The torus $T = \G_m^3$ acts by scaling the linear maps $M_a$, $M_b$, and $M_c$ individually. There are 13 $T$-fixed points (from which we see that the Euler characteristic of $M^\theta$ is 13) which correspond to the following two types of covering quivers:
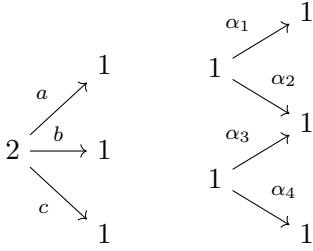
\[
\begin{tikzpicture}[description/.style={fill=white,inner sep=2pt}]
\matrix(m)[matrix of math nodes, row sep=1.5em, column sep=2em, text height=1.5ex, text depth=0.25ex]
{
& 1 \\
2& 1 \\
& 1 \\
};
\path[->, font=\scriptsize]
(m-2-1) edge node[auto] {$a$} (m-1-2)
(m-2-1) edge node[auto] {$b$} (m-2-2)
(m-2-1) edge node[below left] {$c$} (m-3-2)
;
\end{tikzpicture}
\quad \quad
\begin{tikzpicture}[description/.style={fill=white,inner sep=2pt}]
\matrix(m)[matrix of math nodes, row sep=.5em, column sep=2em, text height=1.5ex, text depth=0.25ex]
{
& 1 \\
1& \\
& 1 \\
1& \\
& 1 \\
};
\path[->, font=\scriptsize]
(m-2-1) edge node[auto] {$\alpha_1$} (m-1-2)
(m-2-1) edge node[auto] {$\alpha_2$} (m-3-2)
(m-4-1) edge node[auto] {$\alpha_3$} (m-3-2)
(m-4-1) edge node[auto] {$\alpha_4$} (m-5-2)
;
\end{tikzpicture}
\]
In the right-hand picture $\alpha_1,\dots,\alpha_4 \in \{a,b,c\}$ such that $\alpha_1 \neq \alpha_2 \neq \alpha_3 \neq \alpha_4$ (up to $S_2$-symmetry, there are 12 of those combinations).

We show that in this case there are infinitely many one-dimensional $T$-orbits. For a pair $(x,y) \in k \times k$ let $M(x,y)$ be the representation
\begin{displaymath}
M(x,y) = \left( \begin{pmatrix} 0 & 0 \\ 0 & 1 \\ x & y \end{pmatrix}, \begin{pmatrix} 1 & 0 \\ 0 & 0 \\ 0 & 0 \end{pmatrix}, \begin{pmatrix} 0 & 0 \\ 1 & 0 \\ 0 & 1 \end{pmatrix} \right).
\end{displaymath}
It can be seen that every representation $M(x,y)$ is $\theta$-stable and that the $G_d$-orbits of $M(x,y)$ and $M(x',y')$ are disjoint, provided that $(x,y) \neq (x',y')$. Let $t = (u,v,w) \in T$. Then $t.M(x,y)$ is contained in the $G_d$-orbit of $M((u^2/w^2)x,(u/w)y)$; indeed $g\cdot t.M(x,y) = M((u^2/w^2)x,(u/w)y)$ for
\begin{displaymath}
g = \left( \begin{pmatrix} w & 0 \\ 0 & u \end{pmatrix}, \begin{pmatrix} w/v & 0 & 0 \\ 0 & 1 & 0 \\ 0 & 0 & u/w \end{pmatrix} \right).
\end{displaymath}
For $x,y \in k^\times$, the $T$-orbit of the isomorphism class $[M(x,y)]$, regarded as a point of $M^\theta(Q,d)$, is therefore $T[M(x,y)] = \big\{ \big[M\big(z^2x,zy\big)\big] \,|\,  z \in k^\times \big\}$; in particular, it is one-dimensional. We see that the value $x/y^2$ is a well-defined invariant of the $T$-orbit $T[M(x,y)]$ and it separates all these $T$-orbits. We have thus found a family of one-dimensional $T$-orbits indexed by $k^\times$.

Let $u_1 = \xi_{i,1}$, $u_2 = \xi_{i,2}$ be the equivariant Chern roots of the $G_d \times T$-equivariant bundle $\VV_i$ and let $v_s = \xi_{j,s}$ (with $s = 1,2,3$) be the equivariant Chern roots of $\VV_j$. In $\smash{S(Z_d)^{W_d}_\Q}$, we write $z_{k,l} := \smash{z_{k,l}^{i,j}}$, as we have only two vertices. We just need to consider the elements $z_{1,1}$, $z_{2,1}$, $z_{1,2}$, and $z_{1,3}$ as their images under $f\colon \smash{S(Z_d)_\Q^{W_d}} \to \smash{S(X(PT_d))_\Q^{W_d}}$ span $\smash{S(X(PT_d))_\Q^{W_d}}$ as a ring.

1. We look at the fixed point $[M]$ which corresponds to the covering quiver
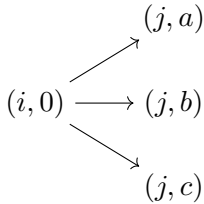
\[
\begin{tikzpicture}[description/.style={fill=white,inner sep=2pt}]
\matrix(m)[matrix of math nodes, row sep=1.5em, column sep=2em, text height=1.5ex, text depth=0.25ex]
{
& (j,a) \\
(i,0)& (j,b) \\
& (j,c) \\
};
\path[->, font=\scriptsize]
(m-2-1) edge (m-1-2)
(m-2-1) edge (m-2-2)
(m-2-1) edge (m-3-2)
;
\end{tikzpicture}
\]
(abusing notation, we write $a$ for the character $x_a$, and so on) and dimension vector $\beta = (2,1,1,1)$.
After an appropriate choice of bases, the characters $\smash{\chi_{i,r}^{(\beta)}}$ are given as $\smash{\chi_{i,1}^{(\beta)}} = 0 = \smash{\chi_{i,2}^{(\beta)}}$, $\smash{\chi_{j,1}^{(\beta)}} = a$, $\smash{\chi_{j,2}^{(\beta)}} = b$, and $\smash{\chi_{j,3}^{(\beta)}} = c$.
We now compute the values $f_\beta(z_{k,l})$ as
\begin{alignat*}{3}
& f_\beta(z_{1,1}) = 2(a+b+c), \qquad && f_\beta(z_{2,1}) = a^2+b^2+c^2,& \\
& f_\beta(z_{1,2}) = 2e_2(a,b,c), \qquad && f_\beta(z_{1,3}) = 2abc.&
\end{alignat*}

2. We consider a fixed point $[M]$ of the second kind. Up to translation the support of the covering dimension vector is
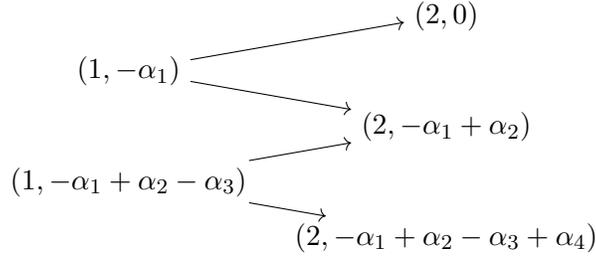
\[
\begin{tikzpicture}[description/.style={fill=white,inner sep=2pt}]
\matrix(m)[matrix of math nodes, row sep=.5em, column sep=1em, text height=1.5ex, text depth=0.25ex]
{
& (2,0) \\
(1,-\alpha_1)& \\
& (2,-\alpha_1+\alpha_2) \\
(1,-\alpha_1+\alpha_2-\alpha_3)& \\
& (2,-\alpha_1+\alpha_2-\alpha_3+\alpha_4) \\
};
\path[->, font=\scriptsize]
(m-2-1) edge (m-1-2)
(m-2-1) edge (m-3-2)
(m-4-1) edge (m-3-2)
(m-4-1) edge (m-5-2)
;
\end{tikzpicture}
\]
In this case the characters $\chi_{i,r}^{(\beta)}$ are $\smash{\chi_{i,1}^{(\beta)}} = -\alpha_1$, $\smash{\chi_{i,2}^{(\beta)}} = -\alpha_1+\alpha_2-\alpha_3$, $\smash{\chi_{j,1}^{(\beta)}} = 0$, $\smash{\chi_{j,2}^{(\beta)}} = -\alpha_1+\alpha_2$, and $\smash{\chi_{j,3}^{(\beta)}} = -\alpha_1+\alpha_2-\alpha_3+\alpha_4$.
The difference $\smash{\chi_{j,t}^{(\beta)}}-\smash{\chi_{i,r}^{(\beta)}}$ is the entry in the t\textsuperscript{th} row and r\textsuperscript{th} column of the following table:
\begin{displaymath}
\begin{array}{r|c|c}
& 1 & 2 \\
\hline
1 & \alpha_1 & \alpha_1-\alpha_2+\alpha_3\\
\hline
2 & \alpha_2 & \alpha_3 \\
\hline
3 & \alpha_2-\alpha_3+\alpha_4 & \alpha_4
\end{array}
\end{displaymath}
Using these, we compute the elements $f_\beta(z_{k,l})$ as
\begin{gather*}
f_\beta(z_{1,1}) = 2\alpha_1+\alpha_2+\alpha_3+2\alpha_4, \\
f_\beta(z_{2,1}) = \alpha_1^2-\alpha_1\alpha_2+\alpha_1\alpha_3+\alpha_2\alpha_3+\alpha_2\alpha_4-\alpha_3\alpha_4 + \alpha_4^2, \\
f_\beta(z_{1,2}) = 2\alpha_1\alpha_2+2\alpha_1\alpha_4+\alpha_2^2-2\alpha_2\alpha_3+\alpha_3^2+2\alpha_3\alpha_4, \\
f_\beta(z_{1,3}) = \alpha_1\alpha_2^2-\alpha_1\alpha_2\alpha_3 + \alpha_1\alpha_2\alpha_4 + \alpha_1\alpha_3\alpha_4-\alpha_2\alpha_3\alpha_4+\alpha_3^2\alpha_4.
\end{gather*}

Let $C = \big\{ (\alpha_1,\dots,\alpha_4) \in \{a,b,c\}^4 \,|\,  \alpha_1 \neq \alpha_2 \neq \alpha_3 \neq \alpha_4 \big\}/S_2$ where $S_2$ reverses the order of the tuple. Denote the element of $C$ represented by $(\alpha_1,\dots,\alpha_4)$ by $\alpha_1\cdots \alpha_4$. Then
\begin{displaymath}
C = \{ abab,abac,abca,abcb,acab,acac,acbc,babc,bacb,bcac,bcbc,cabc \}.
\end{displaymath}
We can now determine the image of the map $\iota^*\colon A_T^*\big(M^\theta(Q,d)\big)_\Q \to \big(\Q \oplus \Q^C\big) \otimes_\Q S(X(T))_\Q \cong \Q[a,b,c]^{13}$. By Corollary~\ref{c:isolated}, it is the subring generated by $(a,\dots,a)^T$, $(b,\dots,b)^T$, $(c,\dots,c)^T$ and the vectors
\begin{gather*}
\begin{pmatrix}
2(a+b+c) \\ 3a+3b \\ 3a+b+2c \\ 4a+b+c \\ 2a+3b+c \\ 3a+2b+c \\ 3a+3c \\ 2a+b+3c \\ a+3b+2c \\ a+4b+c \\ a+2b+3c \\ 3b+3c \\ a+b+4c
\end{pmatrix}, \qquad
\begin{pmatrix}
a^2+b^2+c^2 \\ 2a^2 - ab + 2b^2 \\ 2a^2 - ac + bc + c^2 \\ 2a^2 + bc \\ a^2 - ab + 2b^2 + ac \\ 2a^2 - ab + b^2 + bc \\ 2a^2 - ac + 2c^2 \\ a^2 + ab - ac + 2c^2 \\ 2b^2 + ac - bc + c^2 \\ 2b^2 + ac \\ ab + b^2 - bc + 2c^2 \\ 2b^2 - bc + 2c^2 \\ ab + 2c^2
\end{pmatrix},
\\
\begin{pmatrix}
2ab+2ac+2bc \\
a^2 + 4ab + b^2 \\ a^2 + b^2 + 4ac \\ 2a^2 + 2ab + b^2 + 2ac - 2bc + c^2 \\ 4ab + b^2 + c^2 \\ a^2 + 4ab + c^2 \\ a^2 + 4ac + c^2 \\ b^2 + 4ac + c^2 \\ a^2 + b^2 + 4bc \\ a^2 + 2ab + 2b^2 - 2ac + 2bc + c^2 \\ a^2 + 4bc + c^2 \\ b^2 + 4bc + c^2 \\ a^2 - 2ab + b^2 + 2ac + 2bc + 2c^2
\end{pmatrix}, \qquad
\begin{pmatrix}
2abc \\
a^2b + ab^2 \\ -a^2b + ab^2 + 2a^2c \\ a^2b + ab^2 + a^2c - 2abc + ac^2 \\ 2ab^2 - b^2c + bc^2 \\ 2a^2b - a^2c + ac^2 \\ a^2c + ac^2 \\ b^2c + 2ac^2 - bc^2 \\ a^2b - ab^2 + 2b^2c \\ a^2b + ab^2 - 2abc + b^2c + bc^2 \\ a^2c - ac^2 + 2bc^2 \\ b^2c + bc^2 \\ a^2c - 2abc + b^2c + ac^2 + bc^2
\end{pmatrix}.
\end{gather*}
\end{ex}

\section{Thin quiver moduli}

We consider the special case of an acyclic finite quiver $Q$ and the dimension vector $d = \one := (1,\dots,1)$ (formally $d_i = 1$ for every $i \in Q_0$). In this case the group $G_{\one} = \smash{(\G_m)^{Q_0}}$ is a torus. A representation $M \in R(Q,\one)$ consists of $M_\alpha \in k$; an element $g \in G_{\one}$ acts via $g\cdot M = \smash{\big(g_{t(\alpha)}g_{s(\alpha)}^{-1}M_\alpha\big)_\alpha}$. Again, the action descends to an action of $PG_{\one} = G_{\one}/\Delta$. Let $T = \smash{(\G_m)^{Q_1}}$ which acts, like in the general case, by scaling. The action of $PG_{\one}$ can be recovered from the $T$-action by embedding $PG_{\one}$ as a subtorus via the map
\begin{displaymath}
\sigma_Q\colon \ G_{\one} \to T,\ g \mapsto \big(g_{t(\alpha)}g_{s(\alpha)}^{-1}\big)_\alpha,
\end{displaymath}
whose kernel is precisely $\Delta$; note that we assume $Q$ to be connected. Let $T_0$ be the cokernel of this map. That means we have an exact sequence of tori
\begin{displaymath}
1 \to PG_{\one} \to T \to T_0 \to 1.
\end{displaymath}
Let $\theta$ be a stability condition. Without loss of generality we may assume $\theta(\one) = 0$. As the image of $PG_{\one}$ inside $T$ acts trivially on $M^{\theta-\st}(Q,\one)$ we obtain an action of~$T_0$ on the moduli space. The torus $T_0$ acts with a dense orbit so the moduli space is toric~\cite{AH:99}. By virtue of the exact sequence of tori above~-- which splits~-- we obtain an isomorphism of stacks $\big[M^{\theta-\st}(Q,\one)/T_0\big] \cong \big[R(Q,\one)^{\theta-\st}/T\big]$. Hence, to determine the $T_0$-equivariant Chow ring of $M^{\theta-\st}(Q,\one)$ it suffices to compute the $T$-equivariant Chow ring of $R(Q,\one)^{\theta-\st}$. Using similar arguments as in the third section, we obtain the following characterizations of $A_T^*\big(R(Q,\one)^{\theta-\sst}\big)$ and $A_T^*\big(R(Q,\one)^{\theta-\st}\big)$ in terms of generators and relations:

\begin{prop} \label{p:taut_toric}The kernels of the surjections
\begin{displaymath}
\begin{tikzcd}
\Q[x_\alpha]_{\alpha \in Q_1} = &[-3em] A_{T}^*(R(Q,\one))_\Q \arrow{r}{j_1^*} \arrow{rd}{j_2^*} & A_{T}^*\big(R(Q,\one)^{\theta-\sst}\big)_\Q \arrow{d}{} \\
&& A_{T}^*\big(R(Q,\one)^{\theta-\st}\big)_\Q,
\end{tikzcd}
\end{displaymath}
which are induced by the open embeddings $j_1\colon R(Q,\one)^{\theta-\sst} \to R(Q,\one)$ and $j_2\colon R(Q,\one)^{\theta-\st} \to R(Q,\one)$ are given as follows: for a subset $I \sub Q_0$ let
\begin{displaymath}
x_I := \prod_{\substack{\alpha \in Q_1\\ s(\alpha) \in I,\ t(\alpha) \notin I}} x_\alpha.
\end{displaymath}
Then $\ker(j_1^*)$ is the ideal generated by all $x_I$ with $\theta(\one_I) > 0$ and $\ker(j_2^*)$ is generated by all expressions $x_I$ for which $\theta(\one_I) \geq 0$.
\end{prop}

In the above statement $\one_I \in \Z^{Q_0}$ denotes the characteristic function on the subset $I \sub Q_0$.

It is easy to read off a $\Q$-linear basis from this characterization as the ideal that we are dividing out is generated by monomials. For a tuple $\gamma = (\gamma_\alpha)_{\alpha \in Q_1} \in \smash{\Z_{\geq 0}^{Q_1}}$ write $x^\gamma := \smash{\prod_{\alpha \in Q_1}} \smash{x_\alpha^{\gamma_\alpha}}$. For a subset $I \sub Q_0$ put $J(I) := \{ \alpha \in Q_1 \,|\,  s(\alpha) \in I,\, t(\alpha) \notin I \}$. Then $x_I = \smash{x^{\one_I}}$. A basis of $A_T^*\big(R(Q,\one)^{\theta-\sst}\big)_\Q$ is given by all monomials~$x^\gamma$ where $\operatorname{supp}(\gamma)$ contains no subset $J(I)$ for which $\theta(\one_I) > 0$; the monomials $x^\gamma$ for which $J(I) \nsubseteq \operatorname{supp}(\gamma)$ for all $I \sub Q_0$ with $\theta(\one_I) \geq 0$ form a~basis of $A_T^*\big(R(Q,\one)^{\theta-\st}\big)_\Q$. A basis of $A_T^*\big(R(Q,\one)^{\theta-\sst}\big)_\Q$ can be obtained in a~similar way.

As a next step we would like to determine the pull-back of the embedding of the fixed point locus $M^{\theta-\st}(Q,\one) \to M^{\theta-\st}(Q,\one)^{T_0}$. We introduce the following notion: a subset $H \sub Q_1$ is called a spanning tree if the underlying graph of $(Q_0,H)$ is a tree, which means it is connected and cycle-free.
Introduce the formal symbol $\alpha^{-1}$ for every arrow $\alpha \in Q_1$ and formally define $s\big(\alpha^{-1}\big) = t(\alpha)$ and $t\big(\alpha^{-1}\big) = s(\alpha)$. An unoriented path is a sequence $p = \alpha_r^{\epsilon_r}\cdots\alpha_1^{\epsilon_1}$ such that $\smash{s\big(\alpha_{\nu+1}^{\epsilon_{\nu+1}}\big)} = \smash{t\big(\alpha_\nu^{\epsilon_\nu}\big)}$ for $\nu=1,\dots,r-1$. We define $s(p) = s\big(\alpha_1^{\epsilon_1}\big)$ and $t(p) = t\big(\alpha_r^{\epsilon_r}\big)$. By the spanning tree property there exists for every $i,j \in Q_0$ an unoriented path $p = \alpha_r^{\epsilon_r}\cdots\alpha_1^{\epsilon_1}$ in $H$ such that $s(p) = i$ and $t(p) = j$.

Let $H \sub Q_1$ be any subset. Define the representation $M_H \in R(Q,\one)$ by
\begin{displaymath}
M_\alpha = \begin{cases} 1, & \alpha \in H, \\ 0, & \alpha \notin H. \end{cases}
\end{displaymath}
Note that for a representation $M$ of $Q$ with support $H := \operatorname{supp}(M)$ the representation $M_H$ lies in the same $T$-orbit as $M$ because $T$ acts by scaling along the arrows.

Now assume that $H$ is a spanning tree. We say $H$ is $\theta$-stable if the representation $M_H$ is $\theta$-stable. Denote by $T_{H^c}$ the subtorus of all $t = (t_\alpha)_{\alpha \in Q_1} \in T$ for which $t_\alpha = 1$ whenever $\alpha \in H$. We use $H^c$ as a short-hand for $Q_1 - H$.

\begin{lem} \label{l:fix_toric}
Let $M \in R(Q,\one)^{\theta-\st}$.
\begin{enumerate}\itemsep=0pt
\item[$1.$] The isomorphism class $[M]$ is a fixed point of the $T_0$-action on $M^{\theta-\st}(Q,\one)$ if and only if $\operatorname{supp}(M)$ is a $\theta$-stable spanning tree. In this case $M$ and $M_H$ are isomorphic.
\item[$2.$] Let $H$ be a spanning tree of $Q$. Then the composition $T_{H^c} \to T \to T_0$ is an isomorphism.
\end{enumerate}
\end{lem}

\begin{proof}This can easily be deduced from the description of the fixed point locus in Theorem~\ref{t:weist}.

Denote $\phi\colon T_{H^c} \to T \to T_0$. For $\alpha$ in $Q_1$ let $p = \alpha_r^{\epsilon_r}\cdots \alpha_1^{\epsilon_1}$ be the unique unoriented path in~$H$ such that $s(p) = s(\alpha)$ and $t(p) = t(\alpha)$; note that if $\alpha \in H$ then $p = \alpha$. Define a~morphism $\smash{\tilde{\psi}}\colon T \to T_{H^c}$ by $\smash{\tilde{\psi}}(t) = s$, where
\begin{displaymath}
s_\alpha := t_\alpha \cdot \big(t_{\alpha_r}^{\epsilon_r}\cdots t_{\alpha_1}^{\epsilon_1}\big)^{-1}
\end{displaymath}
for all $\alpha \in Q_1$.
We argue that the morphism $\smash{\tilde{\psi}}$ descends to a morphism $\psi\colon T_0 \to H_{H^c}$. Let $g \in G_\one$. Define $t \in T$ by $t_\alpha = g_{t(\alpha)}g_{s(\alpha)}^{-1}$ (all $\alpha \in Q_1$). We obtain $\smash{\tilde{\psi}}(t) = s$, where
\begin{displaymath}
s_\alpha = g_{t(\alpha)}g_{s(\alpha)}^{-1}\big(\underbrace{g_{t(\alpha_r)}^{\epsilon_r}g_{s(\alpha_r)}^{-\epsilon_r}\dots g_{t(\alpha_1)}^{\epsilon_1}g_{s(\alpha_1)}^{-\epsilon_1}}_{=g_{t(p)}g_{s(p)}^{-1}}\big)^{-1} = 1
\end{displaymath}
for every $\alpha \in Q_1$.
We show that $\phi$ and $\psi$ are mutually inverse. It is easy to see that $\psi\phi$, which is the restriction $\smash{\tilde{\psi}}|_{T_{H^c}}$, is the identity. To show that $\phi\psi = \id_{T_0}$, it is sufficient to prove that for every $t \in T$, the product $r := t\smash{\tilde{\psi}}(t)^{-1}$ lies in the image of $\sigma_Q\colon G_\one \to T$. We compute
\begin{displaymath}
r_\alpha = t_{\alpha_r}^{\epsilon_r}\cdots t_{\alpha_1}^{\epsilon_1}
\end{displaymath}
for all $\alpha \in Q_1$.
So $r_\alpha = t_\alpha$ for every $\alpha \in H$. By the following lemma there exists $g \in G_\one$ such that $t_\alpha = g_{t(\alpha)}g_{s(\alpha)}^{-1}$ for every $\alpha \in H$. For an arbitrary $\alpha \in Q_1$, we then obtain
\begin{displaymath}
r_\alpha = g_{t(p)}g_{s(p)}^{-1} = g_{t(\alpha)}g_{s(\alpha)}^{-1}.
\end{displaymath}
Therefore, $r$ does lie in $\im \sigma_Q$. This proves the second assertion of the lemma.
\end{proof}

\begin{lem}Let $H$ be a quiver whose underlying unoriented graph is a tree. Then the morphism $\sigma_H\colon G_\one \to T_H = \G_m^{H_1}$ is surjective.
\end{lem}

\begin{proof}We argue by induction on the number $n$ of vertices of $H$. The base $n=1$ is trivial. For $n > 1$ choose a leaf $l \in H_0$ of $H$, i.e., a vertex which is adjacent to precisely one other vertex. A leaf exists because the underlying graph of $H$ is a tree. Without loss of generality we assume that $l$ is a sink; the case of a source is completely analogous. Let $\gamma\colon k \to l$ be the unique arrow adjacent to~$l$. Consider the subquiver $H'$ defined by $H'_0 = H_0 - \{l\}$ and $H'_1 = H_1 - \{\gamma\}$. Let $G' := G_{\one_{H'}}$. Then $G_\one = G' \times \G_m$. By induction assumption there exists $g' \in G'$ such that $\smash{g'_{t(\alpha)}{g'}_{s(\alpha)}^{-1}}$ for all $\alpha \in H'_1$. Define $g_l := t_\gamma g'_k$. Then $g := (g',g_l)$ is an inverse image of $t$ under $\sigma_H\colon G_\one \to T_H$.
\end{proof}

We hence may identify for a $\theta$-stable spanning tree $H$ the equivariant Chow ring $\smash{A_{T_0}^*}\!(\{[M_H]\})_\Q$ of the (isolated) fixed point $[M_H] \in \smash{M^{\theta-\st}}(Q,\one)$ with
\begin{displaymath}
A_{T_{H^c}}^*(\pt)_\Q \cong A_T^*\big(R(H,\one)^{\theta-\st}\big)_\Q.
\end{displaymath}
The pull-back of the inclusion $i_H$ of the fixed point then corresponds to the map
\begin{displaymath}
i_H^*\colon \ A_T^*\big(R(Q,\one)^{\theta-\st}\big)_\Q \to A_{T_{H^c}}^*(\pt)_\Q = \Q[x_\alpha]_{\alpha \in H^c},
\end{displaymath}
which is defined by
\begin{displaymath}
i_H^*(x_\alpha) = \begin{cases} x_\alpha, & \alpha \in H^c, \\ 0, & \alpha \in H. \end{cases}
\end{displaymath}

We now want to describe the image of the pull-back of the inclusion of the fixed point locus. We use the following result of Brion. It is the algebraic analog of \cite[Theorem~1.2.2]{GKM:98}.

\begin{thm}[{\cite[Theorem~3.4]{Brion:97}}] \label{t:brion}
Let $X$ be a smooth projective variety on which a torus $T$ acts with finitely many fixed points $x_1,\dots,x_N$ and with finitely many one-dimensional orbits. Then the image of the localization map
\begin{displaymath}
i^*\colon \ A_T^*(X)_\Q \to S(X(T))_\Q^{\oplus N}
\end{displaymath}
consists of all tuples $(f_1,\dots,f_N)$ such that $f_i \equiv f_j$ modulo $\chi$ whenever there exists a one-dimensional $T$-orbit on which $T$ acts with the character~$\chi$ and whose closure contains~$x_i$ and~$x_j$.
\end{thm}

We need to require $Q$ to be acyclic and $\theta$ to be generic for $\one$ in order to ensure the moduli space is smooth and projective. The action of $T_0$ on $M^{\theta}(Q,\one)$ possesses just finitely many one-dimensional orbits. To describe them we introduce the notion of a spanning almost tree. A subset $\Omega \sub Q_1$ is called a spanning almost tree if it is not a spanning tree but there exists an arrow $\alpha \in \Omega$ such that $\Omega-\{\alpha\}$ is a spanning tree. Note that this forces $(Q_0,\Omega)$ to be connected. Given a spanning almost tree we again define a representation $M_\Omega \in R(Q,\one)$ by assigning $M_\alpha = 1$ for $\alpha \in \Omega$ and $M_\alpha = 0$ otherwise. We say $\Omega$ is $\theta$-stable if $M_\Omega$ is $\theta$-stable.

\begin{lem}Let $M \in R(Q,\one)^{\theta}$. The orbit of $[M]$ under $T_0$ is one-dimensional if and only if $\operatorname{supp}(M)$ is a $\theta$-stable spanning almost tree.
\end{lem}

\begin{proof}\looseness=-1 Let $M$ be a $\theta$-stable representation with a one-dimensional orbit. Put $\Omega := \operatorname{supp}(M)$. Note that stability of $M$ implies connectedness of $(Q_0,\Omega)$. Assume that $\Omega$ is not a $\theta$-stable spanning almost tree. As $M_\Omega$ is contained in the $T$-orbit of $M$ we may assume $M = M_\Omega$. Let $B = \{ \alpha_1,\dots,\alpha_r\}$ be a maximal subset of~$\Omega$ such that $(Q_0,\Omega-B)$ is connected. Then~$r \geq 2$ as~$\Omega$ is neither a spanning almost tree nor a spanning tree. Moreover, $H := \Omega-B$ is a spanning tree (which is not necessarily $\theta$-stable). Let $(t_1,\dots,t_r) \in (k^\times)^r$. Define the representation $M_t$ by
\begin{displaymath}
(M_t)_\alpha = \begin{cases} t_k, & \alpha = \alpha_k, \\ 1, & \alpha \in H, \\ 0, & \text{otherwise}. \end{cases}
\end{displaymath}
Clearly $M_t$ is in the $T$-orbit of $M$. Moreover two such representations $M_t$ and $M_s$ are not isomorphic as a $g \in G_1$ with $g \cdot M_t = M_s$ would need to satisfy $g \cdot M_H = M_H$ which implies $g = 1$ by Lemma~\ref{l:fix_toric} (stability is not necessary for this argument to hold). This shows that the orbit of $[M]$ would be (at least) $r$-dimensional which is a contradiction.
The converse direction is obvious.
\end{proof}

The closure of such a one-dimensional orbit is isomorphic to $\P^1$ and hence contains precisely two fixed points. This means that from a $\theta$-stable spanning almost tree there are precisely two ways to remove an arrow and obtain a $\theta$-stable spanning tree. This does not seem to be obvious from the combinatorics of stable spanning almost trees. Let these arrows be $\alpha_0$ and $\alpha_\infty$. Let $\Omega_0 = \Omega - \{ \alpha_0 \}$ and $\Omega_\infty = \Omega - \{\alpha_\infty\}$. These are the two $\theta$-stable spanning trees which represent the two fixed points in the closure of the orbit of $[M_\Omega]$. The inclusions $\Omega_0^c \supseteq \Omega^c \sub \Omega_\infty^c$ give rise to the maps
\begin{displaymath}
A_{T_{\Omega_0^c}}^*(\pt) \xto{}{r_{\Omega,0}} A_{T_{\Omega^c}}^*(\pt) \xot{}{r_{\Omega,\infty}} A_{T_{\Omega_\infty^c}}^*(\pt),
\end{displaymath}
which send $r_{\Omega,0}(x_{\alpha_0}) = 0$ and $r_{\Omega,\infty}(x_{\alpha_\infty}) = 0$ and act as the identity on the other variables. Let us write $f|_{\Omega^c}$ for the image of a function $f \in \smash{A_{T_{\Omega_0^c}}^*(\pt)_\Q} = \smash{\Q[x_\alpha]_{\alpha \in \Omega_0^c}}$ or $f \in \smash{A_{T_{\Omega_\infty^c}}^*(\pt)_\Q} = \smash{\Q[x_\alpha]_{\alpha \in \Omega_\infty^c}}$ under the maps $r_{\Omega,0}$ or $r_{\Omega,\infty}$, respectively.

\begin{thm} \label{t:gkm}
The pull-back of the embedding of the fixed point locus $i\colon M^\theta(Q,\one)^{T_0} \into M^\theta(Q,\one)$ is the map
\begin{displaymath}
i^*\colon \ A_T^*\big(R(Q,\one)^\theta\big)_\Q \to \bigoplus_H \Q[x_\alpha]_{\alpha \in H^c},
\end{displaymath}
which sends $f$ to $(i_H^*(f))_H$. It is injective and its image consists precisely of those tuples $(f_H)_H$ for which $f_{\Omega_0}|_{\Omega^c} = f_{\Omega_\infty}|_{\Omega^c}$ for every $\theta$-stable spanning almost tree $\Omega$.
\end{thm}

The direct sum in the theorem ranges over all $\theta$-stable spanning trees $H$ of $Q$.

\begin{proof}
This is an application of Theorem \ref{t:brion}. Let us show how the statement of our theorem follows from it. Let $\Omega$ be a $\theta$-stable spanning almost tree. Let $C$ be the closure of the $T_0$-orbit of $[M_\Omega]$ in $M^\theta(Q,\one)$. Let $t \in T_{\Omega_0^c}$. We obtain
\begin{displaymath}
(t.M)_\alpha = \begin{cases} t_{\alpha_0}, & \alpha = \alpha_0, \\ 1, & \alpha \in \Omega_0, \\ 0, & \alpha \in \Omega^c. \end{cases}
\end{displaymath}
We identify the $\smash{T_{\Omega_0^c}}$-orbit of $[M_\Omega]$ with $\G_m$ by the entry which corresponds to the arrow $\alpha_0$. Then $T_{\Omega_0^c}$ acts on $\G_m$ by the character $x_{\alpha_0}$. Take the resulting identification of $C$ with $\P^1$. The limit for $t_{\alpha_0} \to 0$ is $[M_{\Omega_0}]$ and the limit for $t_{\alpha_0} \to \infty$ must hence be $[M_{\Omega_\infty}]$. Consider the composition of the isomorphisms of tori $\smash{T_{\Omega_0^c}} \to T_0 \to \smash{T_{\Omega_\infty^c}}$ from Lemma \ref{l:fix_toric}. Call it $\phi$. To determine $\phi$, let $t \in \smash{T_{\Omega_0^c}}$ and find the unique $s \in \smash{T_{\Omega_\infty^c}}$ for which there exists a $g \in G_\one$ such that $g_{t(\alpha)}t_\alpha g_{s(\alpha)}^{-1} = s_\alpha$ for all $\alpha \in Q_1$. The element $g$ is uniquely determined up to scaling. We get four equations:
\begin{gather*}
g_{t(\alpha)}g_{s(\alpha)}^{-1} = 1, \qquad \alpha \in \Omega - \{\alpha_0,\alpha_\infty\}, \\
g_{t(\alpha)}t_\alpha g_{s(\alpha)}^{-1} = s_\alpha \qquad   \alpha \in \Omega^c, \\
g_{t(\alpha_0)}t_{\alpha_0}g_{s(\alpha_0)}^{-1} = 1, \\
g_{t(\alpha_\infty)}g_{s(\alpha_\infty)}^{-1} = s_{\alpha_\infty}.
\end{gather*}
As $\Omega_\infty$ is a spanning tree, $(Q_0,\Omega - \{\alpha_0,\alpha_\infty\})$ is not connected. Let $C_1,C_2 \sub Q_0$ be its connected components. The vertices $i_0 := s(\alpha_0)$ and $j_0 := t(\alpha_0)$ lie in different components, so we may assume $i_0 \in C_1$ and $j_0 \in C_2$. As $\Omega_0$ is connected, $i_\infty := s(\alpha_\infty)$ and $j_\infty := t(\alpha_\infty)$ cannot be contained in the same component. We first consider the case where $i_\infty \in C_1$ and $j_\infty \in C_2$. We get $g_i$ is constant on $C_1$ and on $C_2$. Let its value on $C_1$ be $g_1$ and its value on $C_2$ be $g_2$. Now we obtain
\begin{gather*}
g_2t_{\alpha_0}g_1^{-1}  = 1,\qquad  g_2g_1^{-1} = s_{\alpha_\infty}.
\end{gather*}
This implies $s_{\alpha_\infty} = t_{\alpha_0}^{-1}$ and $g_2g_1^{-1} = t_{\alpha_0}^{-1}$. For $\alpha\colon  i \to j$ in $\Omega^c$ we again have to distinguish three cases. If $i$ and $j$ are in the same component then $s_\alpha = t_\alpha$. If $i \in C_1$ and $j \in C_2$ then we get $s_\alpha = g_2t_\alpha g_1^{-1} = t_{\alpha_0}^{-1}t_\alpha$. Finally if $i \in C_2$ and $j \in C_1$ then we obtain $s_\alpha = g_1t_\alpha g_2^{-1} = t_{\alpha_0}t_\alpha$. We define
\begin{displaymath}
\delta_\alpha := \begin{cases} \hphantom{-}0, & i,j \in C_1 \text{ or } i,j \in C_2, \\ -1, & i \in C_1,\ j \in C_2 ,\\ +1, & i \in C_2,\ j \in C_1. \end{cases}
\end{displaymath}
We get
\begin{displaymath}
\phi(t)_\alpha = \begin{cases} t_{\alpha_0}^{-1}, & \alpha = \alpha_\infty, \\ t_{\alpha_0}^{\delta_\alpha}t_\alpha, & \alpha \in \Omega^c, \\ 1, & \alpha \in \Omega_\infty. \end{cases}
\end{displaymath}
We now deal with the second case where $i_\infty \in C_2$ and $j_\infty \in C_1$. Then $\phi$ looks the same except for $\phi(t)_{\alpha_\infty} = t_{\alpha_0}$. The torus $\smash{T_{\Omega_\infty^c}}$ acts on $C$ by the character $x_{\alpha_\infty}$. But as $\lim\limits_{t_{\alpha_0} \to \infty} \phi(t).[M_\Omega] = \lim\limits_{t_{\alpha_0} \to \infty} t.[M_\Omega] = [M_{\Omega_\infty}]$ we see that $\phi(t)_{\alpha_\infty} = t_{\alpha_0}^{-1}$. So the second case will not occur.

The induced map by $\phi$ on the character lattices is $\phi^*\colon X(\smash{T_{\Omega_\infty^c}}) \to X(\smash{T_{\Omega_0^c}})$ which is given by
\begin{gather*}
\phi^*(x_{\alpha_\infty})  = -x_{\alpha_0},\qquad \phi^*(x_\alpha)  = x_\alpha + \delta_\alpha x_{\alpha_0}
\end{gather*}
for $\alpha \in \Omega^c$. Brion's theorem now tells us that the image of $i^*$ consists precisely of those tuples $(f_H)_H$ for which
\begin{displaymath}
f_{\Omega_0}(x_{\alpha_0},x_\alpha)_{\alpha \in \Omega^c} \equiv f_{\Omega_\infty}(-x_{\alpha_0},x_\alpha+\delta_\alpha x_{\alpha_0})_{\alpha \in \Omega^c}
\end{displaymath}
modulo $x_{\alpha_0}$. This is the same as to say $f_{\Omega_0}|_{\Omega^c} = f_{\Omega_\infty}|_{\Omega^c}$.
\end{proof}

\begin{rem}In \cite{PS:18} Pabiniak and Sabatini find a basis for the $T$-equivariant cohomology ring~$H_T^*(M)$ regarded as a module over $S(X(T))$ for a compact symplectic toric manifold $M$ with torus $T$. Let us briefly recap their result. They fix a moment map and a generic component~$\mu$ of it. A result of Kirwan states that for every $T$-fixed point $p \in M^T$ there exists a class $\nu_p \in H^{2\lambda_p}(M)$ such that its restriction $\nu_p|p \in \smash{H^{2\lambda_p}}(\{p\})$ agrees with the equivariant Euler class of the negative normal bundle $N_p^-$ of~$\mu$ at $p$, while $\nu_p|q = 0$ for every $q \in M^T$ with $\mu(q) < \mu(p)$. Such a class $\nu_p$ is called a Kirwan class. Any collection of Kirwan classes $\{\nu_p\}_{p \in M^T}$ is a basis of~$H_T^*(M)$ over~$S(X(T))$. Kirwan classes, however, are not unique. In \cite[Definition~4.5]{PS:18} a~special type of Kirwan class is defined, the so-called i-canonical class. This is a Kirwan class~$\tau_p$ such that the local index of~$\tau_p$ at~$p$ is~1, whereas the local index at every other fixed point vanishes. In \cite[Theorem~1.2]{PS:18} it is shown that a set $\{\tau_p\}_{p \in M^T}$ of i-canonical classes exists and is unique. It would be interesting to determine the i-canonical classes in the toric case and compare them to our result.
\end{rem}

We end by illustrating this result in two examples.

\begin{ex}Let $Q$ be the $n+1$-Kronecker quiver, $d = \one = (1,1)$ and $\theta = (1,-1)$ as in Example~\ref{e:proj_space}. We have already identified $M^\theta(Q,\one)$ with $\P^n$. The torus $T_0$ is the cokernel $(\G_m)^{n+1}/\G_m$ of the diagonal embedding $\G_m \to \G_m^{n+1}$. So its character group is $X(T_0) = \big\{ \sum_{\nu=0}^n b_\nu x_\nu \,|\,  \sum_{\nu=0}^n b_\nu = 0\big\}$. The ring $S := S(X(T_0))_\Q$ is the subring of $\Q[x_0,\dots,x_n]$ generated by the differences $x_\nu-x_\mu$. The one-dimensional $T_0$-orbits on $\P^n$ are the orbits of the points $\smash{p_{\mu,\nu}} = [0: \dots :0:1:0:\dots:0:1:0:\dots:0]$ with an entry in the $\mu$\textsuperscript{th} and in the $\nu$\textsuperscript{th} position. On that orbit, $T_0$ acts with the character $\pm(x_\nu-x_\mu)$. Brion's description of the image of the localization map (i.e., Theorem~\ref{t:brion}) tells us that a tuple $f = (f_0,\dots,f_n) \in S^{\oplus n+1}$ belongs to $\im i^*$ if and only if $f_\mu \equiv f_\nu$ modulo $x_\nu-x_\mu$ for all $0 \leq \mu < \nu \leq n$.

To match this with the description of Theorem \ref{t:gkm}, we consider the isomorphism
\begin{displaymath}
\bigoplus_{k \neq \nu} \Z x_k \to X(T_0),
\end{displaymath}
which sends $x_k$ to $x_k-x_\nu$.
This provides us with an isomorphism of rings $\psi\colon R := \bigoplus_{\nu=0}^n\! \Q[x_0,\dots,\allowbreak \hat{x_\nu},\dots,x_n] \to S^{\oplus n+1}$. Let $g = (g_0,\dots,g_n) \in R$. We analyze when $\psi(g) = (f_0,\dots,f_n)$ fulfills the condition of Brion's theorem. For $\mu < \nu$ define $\tilde{g}_\mu \in \Q[x_0,\dots,\hat{x_\nu},\dots,x_n]$ as
\begin{displaymath}
\tilde{g}_\mu := g_\mu(x_1-x_\mu,\dots,x_{\mu-1}-x_\mu,x_{\mu+1}-x_\mu,\dots,x_{\nu-1}-x_\mu,-x_\mu,x_{\nu+1}-x_\mu,\dots,x_n-x_\mu).
\end{displaymath}
Then $f_\mu \equiv f_\nu$ modulo $(x_\nu - x_\mu)$ if and only if $\tilde{g}_\mu \equiv g_\nu$ modulo~$x_\mu$. But the reduction of $\tilde{g}_\mu$ modulo~$x_\mu$ in $\Q[x_0,\dots,\hat{x_\nu},\dots,x_n]/(x_\mu) \cong \Q[x_0,\dots,\hat{x_\mu},\dots,\hat{x_\nu},\dots,x_n]$ agrees with the reduction of~$g_\mu$ modulo~$x_\nu$. This shows that the image of $\iota^*$ consists of those $g = (g_0,\dots,g_n) \in R$ such that for all $\mu < \nu$ the images of $g_\mu$ and $g_\nu$ under the maps
\begin{displaymath}
\Q[x_0,\dots,\hat{x_\mu},\dots,x_n] \to \Q[x_0,\dots,\hat{x_\mu},\dots,\hat{x_\nu},\dots,x_n] \ot \Q[x_0,\dots,\hat{x_\nu},\dots,x_n]
\end{displaymath}
agree. This is precisely the description in Theorem \ref{t:gkm}.
\end{ex}

\begin{ex}
Let $Q = K(2,3)$ be the full bipartite quiver with 2 sources and three sinks. A~representation of $Q$ of dimension vector $\one$ is a $(3 \times 2)$-matrix. The structure group $G_\one$ is the torus $\G_m^2 \times \G_m^3$ which acts via
\begin{displaymath}
(g_1,g_2,h_1,h_2,h_3) \cdot \begin{pmatrix} a_{11} & a_{12} \\ a_{21} & a_{22} \\ a_{31} & a_{32} \end{pmatrix} =
\begin{pmatrix} h_1g_1^{-1}a_{11} & h_1g_2^{-1}a_{12} \\ h_2g_1^{-1}a_{21} & h_2g_2^{-1}a_{22} \\ h_3g_1^{-1}a_{31} & h_3g_2^{-1}a_{32} \end{pmatrix}.
\end{displaymath}
We consider the stability condition $\theta = (3,3,-2,-2,-2)$. This stability condition is generic for~$\one$. The resulting moduli space $M^\theta(Q,\one)$ is isomorphic to the blow-up $\operatorname{Bl}_3\big(\P^2\big)$ of $\P^2$ in three points not on a line. This can be seen, for instance, using the results of \cite{FRS:20}: as $\theta$ agrees with the canonical stability condition and ample stability is fulfilled, the moduli space is a Fano variety. Its dimension is two and its Picard rank is 4, so it has to be isomorphic to $\operatorname{Bl}_3\big(\P^2\big)$.

By Proposition \ref{p:taut_toric} we obtain the following description of the $T$-equivariant Chow ring of $R(Q,1)^\theta$ in terms of generators and relations:
\begin{displaymath}
A_T^*\big(R(Q,\one)^\theta\big)_\Q \cong \frac{\Q\left[ \begin{smallmatrix} x_{11} & x_{12} \\ x_{21} & x_{22} \\ x_{31} & x_{32} \end{smallmatrix} \right]}{(x_{j_1i}x_{j_2i} \,|\,  i=1,2,\ 1 \leq j_1 < j_2 \leq 3) + (x_{j1}x_{j2} \,|\,  j=1,2,3)}.
\end{displaymath}
A $\Q$-vector space basis of this algebra is given by all monomials $x^\gamma$ where $\gamma = (\gamma_{ji}) \in M_{3 \times 2}(\Z_{\geq 0})$ is a matrix with at most one non-zero entry in each row and each column.

Now to the fixed points and the one-dimensional orbits of the action of the rank 2 torus $T_0$ on $M^\theta(Q,\one)$. We have to determine the stable spanning trees and the stable spanning almost trees. We describe them in the following picture:
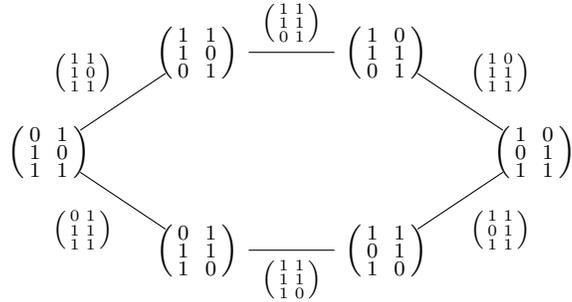
\[
\begin{tikzpicture}[description/.style={fill=white,inner sep=2pt}]
\matrix(m)[matrix of nodes, row sep=2em, column sep=3em, text height=1.5ex, text depth=0.25ex]
{
&[-1.4em] $\left(\begin{smallmatrix} 1 & 1 \\ 1 & 0 \\ 0 & 1 \end{smallmatrix}\right)$& $\left(\begin{smallmatrix} 1 & 0 \\ 1 & 1 \\ 0 & 1 \end{smallmatrix}\right)$&[-1.4em] \\
$\left(\begin{smallmatrix} 0 & 1 \\ 1 & 0 \\ 1 & 1 \end{smallmatrix}\right)$&&& $\left(\begin{smallmatrix} 1 & 0 \\ 0 & 1 \\ 1 & 1 \end{smallmatrix}\right)$ \\
& $\left(\begin{smallmatrix} 0 & 1 \\ 1 & 1 \\ 1 & 0 \end{smallmatrix}\right)$ & $\left(\begin{smallmatrix} 1 & 1 \\ 0 & 1 \\ 1 & 0 \end{smallmatrix}\right)$& \\
};
\path[-, font=\scriptsize]
(m-1-2) edge node[auto] {$\left(\begin{smallmatrix} 1 & 1 \\ 1 & 1 \\ 0 & 1 \end{smallmatrix}\right)$} (m-1-3)
(m-2-1) edge node[auto] {$\left(\begin{smallmatrix} 1 & 1 \\ 1 & 0 \\ 1 & 1 \end{smallmatrix}\right)$}(m-1-2)
(m-1-3) edge node[auto] {$\left(\begin{smallmatrix} 1 & 0 \\ 1 & 1 \\ 1 & 1 \end{smallmatrix}\right)$}(m-2-4)
(m-3-2) edge node[auto] {$\left(\begin{smallmatrix} 0 & 1 \\ 1 & 1 \\ 1 & 1 \end{smallmatrix}\right)$}(m-2-1)
(m-2-4) edge node[auto] {$\left(\begin{smallmatrix} 1 & 1 \\ 0 & 1 \\ 1 & 1 \end{smallmatrix}\right)$}(m-3-3)
(m-3-3) edge node[auto] {$\left(\begin{smallmatrix} 1 & 1 \\ 1 & 1 \\ 1 & 0 \end{smallmatrix}\right)$}(m-3-2)
;
\end{tikzpicture}
\]
The matrices at the vertices of the above graph are the representations $M_H$ which correspond to the 6 stable spanning trees $H$ of $Q$ and the edges are the representations $M_\Omega$ assigned to the 6 stable spanning almost trees $\Omega$ of $Q$. For a stable spanning almost tree $\Omega$, the spanning trees attached to the adjacent vertices correspond to the fixed points which lie in the closure of the one-dimensional orbit associated with $\Omega$. By Theorem~\ref{t:gkm} the pull-back $i^*$ of the embedding of the fixed point locus is the map which is induced by
\begin{gather*}
\Q\left[ \begin{smallmatrix} x_{11} & x_{12} \\ x_{21} & x_{22} \\ x_{31} & x_{32} \end{smallmatrix} \right]  \rightarrow
\Q\left[\begin{smallmatrix} \phantom{x_{32}} & \phantom{x_{32}}\\ & x_{22} \\ x_{31} & \end{smallmatrix}\right]
\oplus   \Q\left[\begin{smallmatrix} x_{11} & \\ & x_{22} \\ \phantom{x_{32}} & \phantom{x_{32}} \end{smallmatrix}\right]
\oplus   \Q\left[\begin{smallmatrix} x_{11} & \\ \phantom{x_{32}} & \phantom{x_{32}} \\ & x_{32} \end{smallmatrix}\right]
\oplus   \Q\left[\begin{smallmatrix} \phantom{x_{32}} & \phantom{x_{32}} \\ x_{21} & \\ & x_{32} \end{smallmatrix}\right]\\
\hphantom{\Q\left[ \begin{smallmatrix} x_{11} & x_{12} \\ x_{21} & x_{22} \\ x_{31} & x_{32} \end{smallmatrix} \right]  \rightarrow}{}
\oplus   \Q\left[\begin{smallmatrix} & x_{12} \\ x_{21} & \\ \phantom{x_{32}} & \phantom{x_{32}} \end{smallmatrix}\right]
\oplus   \Q\left[\begin{smallmatrix} & x_{12} \\ \phantom{x_{32}} & \phantom{x_{32}} \\ x_{31} & \end{smallmatrix}\right],
\\
f \rightarrow\bigg(
f_{\left(\begin{smallmatrix} 1 & 1 \\ 1 & 0 \\ 0 & 1 \end{smallmatrix}\right)},
f_{\left(\begin{smallmatrix} 0 & 1 \\ 1 & 0 \\ 1 & 1 \end{smallmatrix}\right)},
f_{\left(\begin{smallmatrix} 0 & 1 \\ 1 & 1 \\ 1 & 0 \end{smallmatrix}\right)},
f_{\left(\begin{smallmatrix} 1 & 1 \\ 0 & 1 \\ 1 & 0 \end{smallmatrix}\right)},
f_{\left(\begin{smallmatrix} 1 & 0 \\ 0 & 1 \\ 1 & 1 \end{smallmatrix}\right)},
f_{\left(\begin{smallmatrix} 1 & 0 \\ 1 & 1 \\ 0 & 1 \end{smallmatrix}\right)} \bigg),
\end{gather*}
where the components are defined by
\begin{gather*}
f_{\left(\begin{smallmatrix} 1 & 1 \\ 1 & 0 \\ 0 & 1 \end{smallmatrix}\right)} = f\left(\begin{smallmatrix} 0 & 0 \\ 0 & x_{22} \\ x_{31} & 0 \end{smallmatrix}\right),\qquad
f_{\left(\begin{smallmatrix} 0 & 1 \\ 1 & 0 \\ 1 & 1 \end{smallmatrix}\right)} = f\left(\begin{smallmatrix} x_{11} & 0 \\ 0 & x_{22} \\ 0 & 0 \end{smallmatrix}\right),\qquad
f_{\left(\begin{smallmatrix} 0 & 1 \\ 1 & 1 \\ 1 & 0 \end{smallmatrix}\right)} = f\left(\begin{smallmatrix} x_{11} & 0 \\ 0 & 0 \\ 0 & x_{32} \end{smallmatrix}\right), \\
f_{\left(\begin{smallmatrix} 1 & 1 \\ 0 & 1 \\ 1 & 0 \end{smallmatrix}\right)} = f\left(\begin{smallmatrix} 0 & 0 \\ x_{21} & 0 \\ 0 & x_{32} \end{smallmatrix}\right),\qquad
f_{\left(\begin{smallmatrix} 1 & 0 \\ 0 & 1 \\ 1 & 1 \end{smallmatrix}\right)} = f\left(\begin{smallmatrix} 0 & x_{12} \\ x_{21} & 0 \\ 0 & 0 \end{smallmatrix}\right),\qquad
f_{\left(\begin{smallmatrix} 1 & 0 \\ 1 & 1 \\ 0 & 1 \end{smallmatrix}\right)} = f\left(\begin{smallmatrix} 0 & x_{12} \\ 0 & 0 \\ x_{31} & 0 \end{smallmatrix}\right).
\end{gather*}
The image of $i^*$ consist of all tuples
\begin{displaymath}
p =
\left( p_{\left(\begin{smallmatrix} 1 & 1 \\ 1 & 0 \\ 0 & 1 \end{smallmatrix}\right)},
p_{\left(\begin{smallmatrix} 0 & 1 \\ 1 & 0 \\ 1 & 1 \end{smallmatrix}\right)},
p_{\left(\begin{smallmatrix} 0 & 1 \\ 1 & 1 \\ 1 & 0 \end{smallmatrix}\right)},
p_{\left(\begin{smallmatrix} 1 & 1 \\ 0 & 1 \\ 1 & 0 \end{smallmatrix}\right)},
p_{\left(\begin{smallmatrix} 1 & 0 \\ 0 & 1 \\ 1 & 1 \end{smallmatrix}\right)},
p_{\left(\begin{smallmatrix} 1 & 0 \\ 1 & 1 \\ 0 & 1 \end{smallmatrix}\right)} \right),
\end{displaymath}
such that the following six conditions hold:
\begin{gather*}
p_{\left(\begin{smallmatrix} 1 & 1 \\ 1 & 0 \\ 0 & 1 \end{smallmatrix}\right)}\Big( \begin{smallmatrix} \phantom{x_{32}} & \phantom{x_{32}} \\ & x_{22} \\ 0 & \end{smallmatrix} \Big)
 = p_{\left(\begin{smallmatrix} 0 & 1 \\ 1 & 0 \\ 1 & 1 \end{smallmatrix}\right)}\left(\begin{smallmatrix} 0 & \\ & x_{22} \\ \phantom{x_{32}} & \phantom{x_{32}} \end{smallmatrix}\right),\qquad
p_{\left(\begin{smallmatrix} 0 & 1 \\ 1 & 0 \\ 1 & 1 \end{smallmatrix}\right)}\left(\begin{smallmatrix} x_{11} & \\ & 0 \\ \phantom{x_{32}} & \phantom{x_{32}} \end{smallmatrix}\right)
 = p_{\left(\begin{smallmatrix} 0 & 1 \\ 1 & 1 \\ 1 & 0 \end{smallmatrix}\right)}\Big( \begin{smallmatrix} x_{11} & \\ \phantom{x_{32}} & \phantom{x_{32}} \\ & 0 \end{smallmatrix} \Big), \\
p_{\left( \begin{smallmatrix} 0 & 1 \\ 1 & 1 \\ 1 & 0 \end{smallmatrix} \right)}\Big( \begin{smallmatrix} 0 & \\ \phantom{x_{32}} & \phantom{x_{32}} \\ & x_{32} \end{smallmatrix} \Big)
 = p_{\left( \begin{smallmatrix} 1 & 1 \\ 0 & 1 \\ 1 & 0 \end{smallmatrix} \right)}\Big( \begin{smallmatrix} \phantom{x_{32}} & \phantom{x_{32}} \\ 0 & \\ & x_{32} \end{smallmatrix} \Big),\qquad
p_{\left( \begin{smallmatrix} 1 & 1 \\ 0 & 1 \\ 1 & 0 \end{smallmatrix} \right)}\Big( \begin{smallmatrix} \phantom{x_{32}} & \phantom{x_{32}} \\ x_{21} & \\ & 0 \end{smallmatrix} \Big)
 = p_{\left( \begin{smallmatrix} 1 & 0 \\ 0 & 1 \\ 1 & 1 \end{smallmatrix} \right)}\Big( \begin{smallmatrix} & 0 \\ x_{21} & \\ \phantom{x_{32}} & \phantom{x_{32}} \end{smallmatrix} \Big), \\
p_{\left( \begin{smallmatrix} 1 & 0 \\ 0 & 1 \\ 1 & 1 \end{smallmatrix} \right)}\Big( \begin{smallmatrix} & x_{12} \\ 0 & \\ \phantom{x_{32}} & \phantom{x_{32}} \end{smallmatrix} \Big)
 = p_{\left( \begin{smallmatrix} 1 & 0 \\ 1 & 1 \\ 0 & 1 \end{smallmatrix} \right)}\Big( \begin{smallmatrix} & x_{12} \\ \phantom{x_{32}} & \phantom{x_{32}} \\ 0 & \end{smallmatrix} \Big),\qquad
p_{\left( \begin{smallmatrix} 1 & 0 \\ 1 & 1 \\ 0 & 1 \end{smallmatrix} \right)}\Big( \begin{smallmatrix} & 0 \\ \phantom{x_{32}} & \phantom{x_{32}} \\ x_{31} & \end{smallmatrix} \Big)
 = p_{\left( \begin{smallmatrix} 1 & 1 \\ 1 & 0 \\ 0 & 1 \end{smallmatrix} \right)}\Big( \begin{smallmatrix} \phantom{x_{32}} & \phantom{x_{32}} \\ & 0 \\ x_{31} & \end{smallmatrix} \Big).
\end{gather*}
A basis of the image is given by the following elements:
\begin{alignat*}{3}
& i^*(1)= (1,1,1,1,1,1), &&& \\
& i^*\Big(t^{\left(\begin{smallmatrix} 0 & 0 \\ 0 & m \\ 0 & 0 \end{smallmatrix}\right)}\Big)= \big(x_{22}^m,x_{22}^m,0,0,0,0\big),\qquad &&
i^*\Big(t^{\left(\begin{smallmatrix} m & 0 \\ 0 & 0 \\ 0 & 0 \end{smallmatrix}\right)}\Big)  = \big(0,x_{11}^m,x_{11}^m,0,0,0\big), & \\
& i^*\Big(t^{\left(\begin{smallmatrix} 0 & 0 \\ 0 & 0 \\ 0 & m \end{smallmatrix}\right)}\Big)= \big(0,0,x_{32}^m,x_{32}^m,0,0\big),\qquad &&
i^*\Big(t^{\left(\begin{smallmatrix} 0 & 0 \\ m & 0 \\ 0 & 0 \end{smallmatrix}\right)}\Big)  = \big(0,0,0,x_{21}^m,x_{21}^m,0\big), & \\
& i^*\Big(t^{\left(\begin{smallmatrix} 0 & m \\ 0 & 0 \\ 0 & 0 \end{smallmatrix}\right)}\Big)= \big(0,0,0,0,x_{12}^m,x_{12}^m\big),\qquad &&
i^*\Big(t^{\left(\begin{smallmatrix} 0 & 0 \\ 0 & 0 \\ m & 0 \end{smallmatrix}\right)}\Big)= \big(x_{31}^m,0,0,0,0,x_{31}^m\big),& \\
&i^*\Big(t^{\left(\begin{smallmatrix} 0 & 0 \\ 0 & m \\ n & 0 \end{smallmatrix}\right)}\Big)= \big(x_{22}^mx_{31}^n,0,0,0,0,0\big), \qquad&&
i^*\Big(t^{\left(\begin{smallmatrix} n & 0 \\ 0 & m \\ 0 & 0 \end{smallmatrix}\right)}\Big)= \big(0,x_{22}^mx_{11}^n,0,0,0,0\big),& \\
&i^*\Big(t^{\left(\begin{smallmatrix} n & 0 \\ 0 & 0 \\ 0 & m \end{smallmatrix}\right)}\Big)= \big(0,0,x_{32}^mx_{11}^n,0,0,0\big),\qquad &&
i^*\Big(t^{\left(\begin{smallmatrix} 0 & 0 \\ n & 0 \\ 0 & m \end{smallmatrix}\right)}\Big)= \big(0,0,0,x_{32}^mx_{21}^n,0,0\big),& \\
&i^*\Big(t^{\left(\begin{smallmatrix} 0 & m \\ n & 0 \\ 0 & 0 \end{smallmatrix}\right)}\Big)= \big(0,0,0,0,x_{12}^mx_{21}^n,0\big),\qquad &&
i^*\Big(t^{\left(\begin{smallmatrix} 0 & m \\ 0 & 0 \\ n & 0 \end{smallmatrix}\right)}\Big)= \big(0,0,0,0,0,x_{12}^mx_{31}^n\big).&
\end{alignat*}
\end{ex}

\subsection*{Acknowledgements}

I would like to thank Markus Reineke and Silvia Sabatini for inspiring discussions on this subject. I also want to thank the two referees for their valuable comments. At the time this research was conducted I was supported by the DFG SFB/TR 191 ``Symplectic structures in geometry, algebra, and dynamics''.

\pdfbookmark[1]{References}{ref}
\LastPageEnding


\begin{thebibliography}{99}
\footnotesize\itemsep=0pt

\bibitem{AH:99}
Altmann K., Hille L., Strong exceptional sequences provided by quivers,
  \href{https://doi.org/10.1023/A:1009990727521}{\textit{Algebr. Represent. Theory}} \textbf{2} (1999), 1--17.

\bibitem{ASS:06}
Assem I., Skowro\'{n}ski A., Simson D., Elements of the representation theory
  of associative algebras, {V}ol.~1, Techniques of representation theory,
  \textit{London Mathematical Society Student Texts}, Vol.~65, \href{https://doi.org/10.1017/CBO9780511614309}{Cambridge
  University Press}, Cambridge, 2006.

\bibitem{Brion:97}
Brion M., Equivariant {C}how groups for torus actions, \href{https://doi.org/10.1007/BF01234659}{\textit{Transform.
  Groups}} \textbf{2} (1997), 225--267.

\bibitem{CS:74}
Chang T., Skjelbred T., The topological {S}chur lemma and related results,
  \href{https://doi.org/10.2307/1971074}{\textit{Ann. of Math.}} \textbf{100} (1974), 307--321.

\bibitem{EG:98}
Edidin D., Graham W., Equivariant intersection theory, \href{https://doi.org/10.1007/s002220050214}{\textit{Invent. Math.}}
  \textbf{131} (1998), 595--634, \href{https://arxiv.org/abs/alg-geom/9609018}{arXiv:alg-geom/9609018}.

\bibitem{FR:18:Sst_ChowHa}
Franzen H., Reineke M., Semistable {C}how--{H}all algebras of quivers and
  quantized {D}onaldson--{T}homas invariants, \href{https://doi.org/10.2140/ant.2018.12.1001}{\textit{Algebra Number Theory}}
  \textbf{12} (2018), 1001--1025, \href{https://arxiv.org/abs/1512.03748}{arXiv:1512.03748}.

\bibitem{FRS:20}
Franzen H., Reineke M., Sabatini S., Fano quiver moduli, \href{https://arxiv.org/abs/2001.10556}{arXiv:2001.10556}.

\bibitem{GKM:98}
Goresky M., Kottwitz R., MacPherson R., Equivariant cohomology, {K}oszul
  duality, and the localization theorem, \href{https://doi.org/10.1007/s002220050197}{\textit{Invent. Math.}} \textbf{131}
  (1998), 25--83.

\bibitem{King:94}
King A.D., Moduli of representations of finite-dimensional algebras,
  \href{https://doi.org/10.1093/qmath/45.4.515}{\textit{Quart.~J. Math. Oxford Ser.~(2)}} \textbf{45} (1994), 515--530.

\bibitem{Kresch:99}
Kresch A., Cycle groups for {A}rtin stacks, \href{https://doi.org/10.1007/s002220050351}{\textit{Invent. Math.}} \textbf{138}
  (1999), 495--536, \href{https://arxiv.org/abs/math.AG/9810166}{arXiv:math.AG/9810166}.

\bibitem{LP:90}
Le~Bruyn L., Procesi C., Semisimple representations of quivers, \href{https://doi.org/10.2307/2001477}{\textit{Trans.
  Amer. Math. Soc.}} \textbf{317} (1990), 585--598.

\bibitem{MRV:06}
Molina~Rojas L.A., Vistoli A., On the {C}how rings of classifying spaces for
  classical groups, \textit{Rend. Sem. Mat. Univ. Padova} \textbf{116} (2006),
  271--298, \href{https://arxiv.org/abs/math.AG/0505560}{arXiv:math.AG/0505560}.

\bibitem{GIT:94}
Mumford D., Fogarty J., Kirwan F., Geometric invariant theory, 3rd~ed.,
  \textit{Ergebnisse der Mathematik und ihrer Grenzgebiete~(2)}, Vol.~34, \href{https://doi.org/10.1007/978-3-642-57916-5}{Springer-Verlag}, Berlin, 1994.

\bibitem{PS:18}
Pabiniak M., Sabatini S., Canonical bases for the equivariant cohomology and
  {K}-theory rings of symplectic toric manifolds, \href{https://doi.org/10.4310/JSG.2018.v16.n4.a8}{\textit{J.~Symplectic Geom.}}
  \textbf{16} (2018), 1117--1165, \href{https://arxiv.org/abs/1503.04730}{arXiv:1503.04730}.

\bibitem{Reineke:08}
Reineke M., Moduli of representations of quivers, in Trends in representation
  theory of algebras and related topics, \textit{EMS Ser. Congr. Rep.}, \href{https://doi.org/10.4171/062-1/14}{Eur. Math. Soc.},
  Z\"urich, 2008, 589--637, \href{https://arxiv.org/abs/0802.2147}{arXiv:0802.2147}.

\bibitem{RSW:12}
Reineke M., Stoppa J., Weist T., M{PS} degeneration formula for quiver moduli
  and refined {GW}/{K}ronecker correspondence, \href{https://doi.org/10.2140/gt.2012.16.2097}{\textit{Geom. Topol.}}
  \textbf{16} (2012), 2097--2134, \href{https://arxiv.org/abs/1110.4847}{arXiv:1110.4847}.

\bibitem{RW:13}
Reineke M., Weist T., Refined {GW}/{K}ronecker correspondence, \href{https://doi.org/10.1007/s00208-012-0778-0}{\textit{Math.
  Ann.}} \textbf{355} (2013), 17--56, \href{https://arxiv.org/abs/1103.5283}{arXiv:1103.5283}.

\bibitem{RW:20}
Rupel D., Weist T., Cell decompositions for rank two quiver {G}rassmannians,
  \href{https://doi.org/10.1007/s00209-019-02379-6}{\textit{Math.~Z.}} \textbf{295} (2020), 993--1038, \href{https://arxiv.org/abs/1803.06590}{arXiv:1803.06590}.

\bibitem{Weist:13}
Weist T., Localization in quiver moduli spaces, \href{https://doi.org/10.1090/S1088-4165-2013-00436-3}{\textit{Represent. Theory}}
  \textbf{17} (2013), 382--425, \href{https://arxiv.org/abs/0903.5442}{arXiv:0903.5442}.

\end{thebibliography}
\end{document}